\documentclass[11pt,reqno]{amsart}
\usepackage{amssymb,mathrsfs,graphicx,enumerate,enumitem}
\usepackage{amsmath,amsfonts,amssymb,amscd,amsthm,bbm}
\usepackage{algorithm,algorithmic,bm,caption,colortbl,epsfig,extpfeil,float,graphicx,subcaption}
\usepackage[margin=1in]{geometry}
\title[GWP of slightly supercritical SQG and gradient estimate]{Global well-posedness of slightly supercritical SQG equations and gradient estimate}

\author[Choi]{Hyungjun Choi}
\address[Hyungjun Choi]{\newline Department of Mathematics \newline Princeton University, Princeton NJ 08544, USA}
\email{hyungjun.choi@princeton.edu}

\newtheorem{theorem}{Theorem}
\newtheorem{proposition}[theorem]{Proposition}
\newtheorem{corollary}[theorem]{Corollary}
\newtheorem{lemma}[theorem]{Lemma}

\theoremstyle{definition}

\theoremstyle{remark}
\newtheorem*{remark}{Remark}

\numberwithin{equation}{section}

\newcommand{\R}{\mathbb{R}}
\newcommand{\abs}[1]{\left\lvert #1 \right\rvert}
\newcommand{\norm}[1]{\lVert #1 \rVert}

\begin{document}

\date{\today}

\begin{abstract}
We prove the global regularity of smooth solutions for a dissipative surface quasi-geostrophic equation with both velocity and dissipation logarithmically supercritical compared to the critical equation. By this, we mean that a symbol defined as a power of logarithm is added to both velocity and dissipation terms to penalize the equation's criticality. Our primary tool is the nonlinear maximum principle which provides transparent proofs of global regularity for nonlinear dissipative equations. Combining the nonlinear maximum principle with a modulus of continuity, we prove a uniform-in-time gradient estimate for the critical and slightly supercritical surface quasi-geostrophic equation. It improves the previous double exponential bound by Kiselev-Nazarov-Volberg to the single exponential. In addition, we prove eventual exponential decay of the solutions.
\end{abstract}
\maketitle

\section{Introduction}
The global well-posedness of the dissipative surface quasi-geostrophic (SQG) equation \eqref{eq:1.1} has been widely studied in recent literature.
\begin{equation} \label{eq:1.1}
\left\{\begin{array}{rl} \partial_t \theta + (u\cdot \nabla) \theta + \Lambda^\alpha \theta & = 0, \\ u & = \nabla^\perp \Lambda^{-1} \theta, \end{array}\right.
\end{equation}
where $\Lambda = (-\Delta)^{\frac{1}{2}}$. The SQG equation first appeared in the mathematical literature in \cite{Con94}. Due to the positivity of the operator $\Lambda^\alpha$, it can be easily seen that the $L^p$ norms of $\theta$ are nonincreasing in time under smooth evolution. Indeed, the equation has $L^\infty$-maximum principle \cite{Cor04}. The situation differs by whether the order of the diffusion term is higher/identical/lower than the nonlinear transport term under $L^\infty$-rescaling of a solution. The case $\alpha=1$ is termed ``critical'' since the order is identical. The weak formulation and the global well-posedness for the subcritical case ($\alpha>1$) were understood in the usual sense \cite{Con99, Res95}. The global well-posedness in the critical case was first proved in \cite{Caf10} (in $\mathbb{R}^n$) and \cite{Kis07,Kis10} (in $\mathbb{T}^n$). In particular, Kiselev-Nazarov-Volberg \cite{Kis07} analyzed a breakthrough scenario of a modulus of continuity, which they called the nonlocal maximum principle and is studied further in \cite{Kis11}. Shortly after that, Constantin-Vicol \cite{Con12} presented another proof. They introduced so-called nonlinear maximum principle, which can be applied to well-posedness proofs for various active scalar equations.

In the supercritical case ($\alpha<1$), the global well-posedness for generic initial remains an open problem. There have been several results regarding conditional regularity \cite{Con08, Don09b} (a Hölder continuous solution is smooth) and eventual regularization of solutions \cite{Sil10, Dab11, Kis11}. The global regularity has been studied if velocity or dissipation is slightly supercritical. In particular, the slightly supercritical velocity was studied in \cite{Dab12}, where the veloctiy is obtained from $\theta$ by a Fourier multiplier with symbol $i\zeta^\perp \abs{\zeta}^{-1} m(\zeta)$. Here, $m(\zeta)$ grows slower than $\log\log\abs{\zeta}$ as $\zeta \to \infty$. Moreover, the slightly supercritical dissipation for several active scalar equations was studied in \cite{Dab14}, where the dissipation term is given by multiplier with behavior $P(\zeta) \sim \abs{\zeta} (\log \abs{\zeta})^{-\alpha}$ ($0\leq \alpha\leq 1$) for large $\zeta$. These proofs were based on the nonlocal maximum principle introduced in \cite{Kis07, Kis11}.

In this paper, we consider the SQG equation which is slightly supercritical in both velocity and dissipation:
\begin{equation} \label{eq:1.2}
\left\{\begin{array}{rl} \partial_t \theta + (u\cdot \nabla) \theta + \mathcal{L} \theta & = 0, \qquad x\in \R^2,\, t>0, \\ u &= \nabla^\perp \Lambda^{-1} m (\Lambda) \theta, \\
\theta(t=0) & = \theta_0.\end{array}\right.
\end{equation}
The nonlocal operator $\mathcal{L}$ is defined by a convolutional kernel as follows:
\[\mathcal{L}\theta(x) = P.V. \int_{\R^2} (\theta(x) - \theta(y)) \frac{k(x-y)}{\abs{x-y}^2} \,d y.\]
Note that $\mathcal{L} = \Lambda$ when $k(x) = \frac{1}{2\pi \abs{x}}$. The main example we have in mind is
\begin{equation} \label{eq:1.3}
k(r) = \frac{1}{r(\log(10+r^{-1}))^{\alpha_1}},\quad m(\zeta) = (\log(10+ \abs{\zeta}))^{\alpha_2},
\end{equation}
for some $\alpha_1 ,\alpha_2 \geq 0$. Throughout the paper, we assume on $k$ and $m$ as follows:
\begin{itemize} \label{assum:1} \setlength\itemsep{2ex}
	\item The kernel $k(x) = k(\abs{x})$ is radially symmetric, nonnegative, and nonincreasing. In addition, $r^2 k(r)$ is nondecreasing for $r>0$.
	\item (Slightly supercitical dissipation) For any $\epsilon>0$, 
	\[r^{1-\epsilon} k(r)\]
	is nonincreasing for $0<r<r_\epsilon$ and tends to $\infty$ as $r\to 0$.
	\item There exists a constant $c\geq 1$ such that
	\[k \left( \frac{r}{2} \right) \leq c\, k(r),\quad r > 0. \]
	\item The Fourier multiplier $m(\zeta) = m(\abs{\zeta})\geq 1$ is radially symmetric,  nondecreasing, and satisfies the Hörmander-Mikhlin condition:
	\begin{equation*}
	\abs{\zeta}^k |\nabla^k m(\zeta)| \leq C_k m(\zeta),\quad \zeta\neq 0,
	\end{equation*}
	for some constant $C_k > 0$.
	\item (Slightly supercritical velocity) For any $\epsilon>0$,
	\[\abs{\zeta}^{\epsilon} m(\abs{\zeta}^{-1})\]
	is nondecreasing in $0<\abs{\zeta}<r_\epsilon$ and tends to $0$ as $\abs{\zeta}\to 0$.
\end{itemize}
Especially, we call \eqref{eq:1.2} slightly supercritical if $k$ and $m$ satisfy the slightly supercritical-ness conditions (the second and last assumptions above).

We prove the global regularity for a solution of \eqref{eq:1.2} with help of the nonlinear maximum principle by Constantin-Vicol \cite{Con12}. We show a conditional regularity result in the first step, then we prove that required regularity is conserved by evolution in the second step. The main difference between our proof and \cite{Con12} is that we provide a modulus of continuity while \cite{Con12} argues only with uniform continuity (the only small shock condition). We say a function $f$ admits a modulus of continuity $\omega$ if $|f(x) - f(y)| \leq \omega(\abs{x-y})$ for all $x\neq y$. A modulus of continuity function $\omega:[0,\infty)\to [0,\infty)$ should be non-decreasing, $\omega(0) = 0$, and continuous at $0$. An explicit modulus of continuity for solution allows us to construct an exponential gradient estimate; Main Theorem 2.

In section \ref{sec:2}, we list some estimates regarding the relationship between the kernel and the Fourier multiplier. We provide suitable conditions for $k$ and $m$ to behave similarly as the main example \eqref{eq:1.3}. We also collect some elementary integral estimates for later use. In section \ref{sec:3}, we provide conditional regularity result for \eqref{eq:1.2}. If $\theta$ is a bounded weak solution to \eqref{eq:1.2} in time $[0,T]$ and $\theta(\cdot,t)$ admits a modulus of continuity $\omega$ for $t\in [0,T]$ which satisfies
\begin{equation} \label{eq:1.4}
\lim_{r\to 0+} \frac{m(r^{-1}) \omega(r)}{r k(r)} = 0,
\end{equation}
then $\theta$ is a smooth solution. In section \ref{sec:4}, we prove that a modulus of continuity $M \omega$ for $\theta(\cdot,t)$ is conserved by $\eqref{eq:1.2}$ if a constant $M$ depends on $\theta_0$ is sufficiently large and
\begin{equation} \label{eq:1.5}
\lim_{r\to 0+} \frac{m(r^{-1})}{k(r)} \frac{\omega'(r)}{\omega(r)} = 0.
\end{equation}
Combining two steps, we prove the following:

\noindent\textbf{Main Theorem 1.} Assume that $\theta_0 \in \mathcal{S}(\R^2)$. Suppose there is a non-decreasing, continuous, concave function $\omega:[0,\infty)\to[0,\infty)$ satisfying $\omega(0)=0$, \eqref{eq:1.4}, \eqref{eq:1.5}, and for some $\epsilon>0$,
\begin{equation} \label{eq:1.6}
\lim_{r\to 0+} \frac{\omega(r)}{r^{\epsilon}} = \infty, \quad \text{and}\quad \lim_{r\to \infty} \frac{\omega(r)}{r^\epsilon} = 0.
\end{equation}
Then there exists a global smooth solution $\theta$ of the slightly supercritical SQG equation \eqref{eq:1.2}.

\noindent A typical example is a logarithmic multiplier \eqref{eq:1.3}.
\begin{corollary} \label{cor:1}
Assume that $\theta_0 \in \mathcal{S}(\R^2)$. If
\[\alpha_1 + \alpha_2 < 1,\]
then there exists a global smooth solution $\theta$ of
\begin{equation} \label{eq:1.7}
\left\{\begin{array}{rl} \partial_t \theta + (u\cdot \nabla) \theta + \Lambda (\log^{-\alpha_1} \hspace{-.5ex} \Lambda) \theta & = 0, \\ u & = \nabla^\perp \Lambda^{-1} (\log^{\alpha_2} \hspace{-.3ex} \Lambda) \theta, \\ \theta(t=0) & = \theta_0, \end{array}\right.
\end{equation}
where $\log^{\gamma} \Lambda$ is Fourier multiplier opertor with symbol $\left(\log(10+\abs{\xi})\right)^\gamma$.
\end{corollary}

\begin{remark}
Suppose we set $\mathcal{L} = \Lambda (\log^{-\alpha_1} \Lambda)$, i.e., Fourier multiplier with symbol
\[P(\zeta) = \zeta \left(\log(10+\abs{\zeta}) \right)^{-\alpha_1},\quad \alpha_1 > 0.\]
Then Lemma \ref{lemma:6} in section \ref{sec:2} shows that it is same as considering
\[k(r) \sim \frac{1}{r\left( \log (10+ r^{-1} ) \right)^{\alpha_1}}.\]
\end{remark}

This result improves \cite{Dab12} which studied a slightly supercritical velocity, where $m(\zeta)$ being growing slower than double logarithm. Moreover, we prove the global regularity when both velocity and dissipation are slightly supercritical. Hence this paper extends \cite{Dab14}, which studied the slightly supercritical dissipation.

In section \ref{sec:5}, we prove an exponential gradient estimate for the critical and slightly supercritical SQG equation. The critical SQG equation reads as follows:
\begin{equation} \label{eq:1.8}
\left\{\begin{array}{rl} \partial_t \theta + (u\cdot \nabla) \theta + \Lambda \theta & = 0,\\
u & = \nabla^\perp \Lambda^{-1} \theta,\\
\theta(t=0) &= \theta_0 \end{array}\right.
\end{equation}

\noindent\textbf{Main Theorem 2.}
The critical SQG equation \eqref{eq:1.8} with initial $\theta_0 \in \mathcal{S}(\R^2)$ has a unique global smooth solution, and
\begin{equation} \label{eq:1.9}
\norm{\nabla \theta}_{L^\infty_{t,x}} \leq C \norm{\nabla \theta_0}_{L^\infty} \exp(C \norm{\theta_0}_{L^\infty}^{\gamma}) \quad \text{for any } \gamma > 1,
\end{equation}
where constants depend only on $\gamma$. In addition, for the logarithmically supercritical equation \eqref{eq:1.7}, similar gradient estimate holds. See Corollary \ref{cor:12} on section \ref{sec:5} for further discussion.

Our result \eqref{eq:1.9} improves the following double exponential estimate \eqref{eq:1.10}. Kiselev-Nazarov-Volberg \cite{Kis07} used a modulus of continuity $\omega(r)$, which behaves like $\log \log (r/\delta)$ for $\delta \ll r \lesssim 1$. While we use a modulus of continuity $|\log r|^{-\beta}$ for $r \lesssim 1$ with some $\beta>0$ and the method of proving global regularity is different. A similar exponential estimates were available in \cite{Mia12} for modified SQG equations while there method is a direct application of nonlocal maximum principle from Kiselev-Nazarov-Volberg. A major aspect of our method is the conditional regularity result, namely Theorem \ref{thm:8}. It asserts that some modulus of continuity which is even weaker than any Hölder continuity is sufficient for the propagation of higher regularity including $C^\infty$.

\begin{proposition}[Global well-posedness of critical SQG, \cite{Kis07}]
The critical SQG equation \eqref{eq:1.8} with initial $\theta_0\in C^\infty(\mathbb{T}^2)$ has a unique global smooth solution. In addition,
\begin{equation} \label{eq:1.10}
\norm{\nabla \theta}_{L^\infty_{t,x}} \leq C \norm{\nabla \theta_0}_{L^\infty} \exp(\exp(C \norm{\theta_0}_{L^\infty})).
\end{equation}
\end{proposition}

In section \ref{sec:6}, we prove that every subcritical Sobolev norm eventually decays. It is proved in the spirit of \cite{Con01, Cor04, Don10}.

\noindent\textbf{Main Theorem 3.} Consider a slightly supercritical SQG equation \eqref{eq:1.2} satisfying conditions of Main Theorem 1. Assume that the Fourier multiplier $P(\zeta) = P(\abs{\zeta})$ corresponding to the dissipation operator $\mathcal{L}$ satisfies
\[P(\zeta) \geq C k(\abs{\zeta}^{-1}) \geq C_\beta \abs{\zeta}^{\beta},\]
for any $0\leq \beta<1$. Suppose $\theta$ is a global smooth solution to the equation, then
\[\norm{\theta(\cdot,t)}_{\dot{H}^{s}} \leq \exp(- c(t-t_0)) \norm{\theta(\cdot,t_0)}_{\dot{H}^s},\quad t>t_0,\]
for any $s >1$, where $t_0 > 0$ depends only on $\theta_0$ and $c>0$ depends only on $\theta_0$ and $s$.

For the critical SQG equation \eqref{eq:1.8}, the decay of solutions was proved in various ways. It is known for the dissipative SQG equation \eqref{eq:1.1} for $0<\alpha\leq 1$ that
\begin{enumerate}
	\item $\norm{\theta(\cdot,t)}_{L^\infty} \leq C \norm{\theta_0}_{L^2} t^{-\frac{1}{\alpha}}$.
	\item $\norm{\theta}_{L^\infty \big([-\frac{1}{2}, 0]; C^{1,\gamma}(B_{\frac{1}{2}}) \big)} \leq C(\norm{\theta}_{L^\infty([-1,0]; C^{1-\alpha+\gamma}(B_1))}) \norm{\theta}_{L^\infty([-1,0];L^\infty(B_1))}$.
\end{enumerate}
The first assertion can be found in \cite{Cor04, Con09, Sil10} or for the critical case it is Theorem 1 in \cite{Caf10}. The second one is the application of Theorem 1.1 in \cite{Sil12}. For the critical case $\alpha=1$, the global regularity results are known and implies uniform-in-time bound on Hölder norms. Combining two results gives decay of $\norm{\theta(\cdot,t)}_{C^{1,\gamma}}$ as $t\to \infty$. Note that it do not mean that the norms decrease from the initial time or any specified time. Hence these decay estimates do not strictly improves our uniform-in-time estimate, namely Main Theorem 2.

Furthermore, a solution of the critical SQG exhibits exponential decay of Sobolev norms.
\begin{proposition}[Eventual decay of Sobolev norms, Theorem 2.7 in \cite{Don10}]
The critical SQG equation \eqref{eq:1.8} with initial $\theta_0\in \dot{H}^1(\mathbb{T}^2)$ has a unique global smooth solution. In addition,
\[\norm{\theta(\cdot,t)}_{\dot{H}^{1+\beta}(\mathbb{T}^2)} \leq C e^{-t/4} t^{-\beta},\quad t>0,\]
for any $\beta\geq 0$, where $C$ depends only on $\theta_0$ and $\beta$.
\end{proposition}
In addition to the above, the decay estimates can be found in \cite{Abi08} as well.

Lastly, the presented results can be easily reproduced for other $L^\infty$-critical active scalar equations with nonlocal dissipation and a linear balance law $u = T \theta$.

\section{Preliminaries} \label{sec:2}
The local regularity result for the SQG-type equation is standard. It is well-known that $\norm{\nabla \theta}_{L^1_t L^\infty_x}$ is a blow-up criterion.

\begin{proposition}[Local existence of smooth solution, Proposition 2.1 in \cite{Dab12}, \cite{Don09a, Don09b, Don10}] \label{prop:4}
Suppose $\theta_0 \in H^{s}(\R^2)$ and $s>1$. Then there exists $T = T(\norm{\theta_0}_{H^s})>0$ and a weak solution $\theta\in C([0,T]; H^s)$ of slightly supercritical SQG equation \eqref{eq:1.2}. Recall that \eqref{eq:1.2} is called slightly supercritical if $rk(r) = o(r^{\epsilon})$ as $r\to 0$ and $m(\zeta)= o(\abs{\zeta}^{\epsilon})$ as $\abs{\zeta}\to \infty$, for all $\epsilon > 0$. Moreover, if a weak solution $\theta$ satisfies
\[\norm{\nabla \theta}_{L^1_t L^\infty_x} = \int_0^T \norm{\nabla \theta(\cdot,t)}_{L^\infty} < \infty,\]
then $\theta\in C^\infty((0,T]\times \R^2)$.
\end{proposition}

The local well-posedness result for $H^s$ with $s>1$ is suboptimal and can be proved in a standard way. The optimal results, the local well-posedness for critical Sobolev spaces ($H^{2-\alpha}$ for \eqref{eq:1.1}) was proved in \cite{Miu07, Don10}. The stated conditional regularity result is known for standard dissipative SQG equations, see e.g. Theorem 3.5 in \cite{Don09b} or Proposition 4.2 in \cite{Mia12}. The addition of Fourier multipliers, which grows (or decays) slower than any power $\abs{\zeta}^\epsilon$, do not affect its proof.

The construction of global weak solutions to the dissipative SQG equation \eqref{eq:1.1} had been stuided in \cite{Res95}. For each initial $\theta_0\in L^2$ and arbitrary $T>0$, there exists a weak solution $\theta\in C_w([0,T]; L^2) \cap L^2([0,T]; \dot{H}^\frac{\alpha}{2})$ which satisfies the energy inequality. Moreover, $\norm{\theta(\cdot,t)}_{L^p} \leq \norm{\theta_0}_{L^p}$ for any $1<p\leq\infty$. We follow the notion of bounded weak solution in \cite{Con12}, that is $\theta\in L^\infty([0,T]; L^p \cap L^\infty)$ for some $1\leq p<\infty$.

Next, we provide assumptions on $k$ and $m$ and arrange some estimates on the kernel of the operator $\nabla \Lambda^{-1} m(\Lambda)$ and the kernel $r^{-2} k(r)$. The following lemmas reveal relations between an operator's kernel and the Fourier multiplier. Lemma \ref{lemma:5} deals with velocity $u$ and the Lemma \ref{lemma:6} deals with dissipation term $\mathcal{L} \theta$.

\begin{lemma}[Lemma 4.1 of \cite{Dab12}] \label{lemma:5}
Suppose $K$ is the kernel corresponding to the operator $\partial_j \Lambda^{-1} m(\Lambda)$ and $m$ satisfies the prescribed conditions on page \pageref{assum:1}. In other words, $m(\xi)$ being the Hörmander-Mikhlin type and it increases slower than any power $\abs{\xi}^\epsilon$ as $\xi \to \infty$. Then
\[\abs{K(x)} \leq C\abs{x}^{-d} m(\abs{x}^{-1}),\]
and
\[\abs{\nabla K(x)} \leq C \abs{x}^{-d-1} m(\abs{x}^{-1}),\]
for all $x\neq 0$.
\end{lemma}

\begin{lemma}[\cite{Dab14,Ste70}] \label{lemma:6}
Suppose $P(\zeta) = P(\abs{\zeta})$ is a radially symmetric function that is smooth, nonnegative, nondecreasing from zero, $P(0)=0$, and $P(\zeta) \to \infty$ as $\abs{\zeta} \to \infty$. In addition, assume the following for $P$:
\begin{itemize} \setlength\itemsep{2ex}
	\item There is a constant $c\geq 1$ so that $P(2\zeta) \leq c P(\zeta)$ for all $\zeta\in\R^d$.
	\item $P$ is of the Hörmander-Mikhlin type:
	\[\abs{\zeta}^k |\nabla^k P(\zeta)| \leq C_k P(\zeta),\quad \zeta\neq 0,\]
	for some constant $C_k$.
	\item $P$ satisfies growth condition
	\[\int_0^1 P(\abs{\zeta}^{-1}) \abs{\zeta} \,d\zeta < \infty.\]
\end{itemize}
Then, the corresponding radially symmetric kernel $K = \mathcal{F}^{-1} P$ satisfies
\[\abs{K(y)} \leq C \abs{y}^{-d} P(\abs{y}^{-1}),\]
and
\[\abs{\nabla K(y)} \leq C \abs{y}^{-d-1} P(\abs{y}^{-1}),\]
for all $y\not=0\in\R^d$. Moreover, if $P$ satisfies
\begin{itemize}
	\item $\displaystyle (-\Delta)^{\frac{d}{2} + 1} P(\zeta) \geq c \abs{\zeta}^{-d-2} P(\zeta)$
\end{itemize}
for some constant $c>0$. Then, $K$ is bounded below as
\[K(y) \geq c \abs{y}^{-d} P(\abs{y}^{-1}),\]
for all sufficiently small $y$.
\end{lemma}

Lastly, we denote some technical integral inequalities which we will use frequently.
\begin{lemma} \label{lemma:7}
Suppose $k$ and $m$ satisfy the prescribed conditions and $\omega:[0,\infty)\to [0,\infty)$ is a nondecreasing, continuous function with $\omega(0)=0$. In addition, assume that there exists $\epsilon>0$ such that
\begin{equation*} 
\lim_{r\to 0+} \frac{\omega(r)}{r^{\epsilon}} = \infty, \quad \text{and}\quad \lim_{r\to \infty} \frac{\omega(r)}{r^\epsilon} = 0.
\end{equation*}
For sufficiently small $r$, namely $0<r<r_0$, the following inequalities hold.
{\everymath{\displaystyle} \begin{enumerate}[label=(\roman*)]
	\item $\int_r^\infty \frac{k(\rho)}{\rho} \,d\rho \geq \frac{k(r)}{2}$.
	\item $\int_{r/2}^{r} \frac{\omega(\rho) k(\rho)}{\rho} \,d \rho \leq C \omega(r) k(r)$ for a constant $C>0$.
	\item $\int_{r}^{\infty} \omega(\rho) \left(-\frac{k'(\rho)}{\rho} + \frac{2 k(\rho)}{\rho^2} \right) \,d \rho \leq \frac{C \omega(r) k(r)}{r}$ for a constant $C>0$.
	\item $\int_0^r \frac{m(\rho^{-1})^2}{\rho k(\rho)} \,d \rho \leq \frac{4 m(r^{-1})^2}{k(r)}$.
	\item $\int_r^\infty \frac{\omega(\rho) m(\rho^{-1})}{\rho^2} \,d \rho \leq \frac{C \omega(r) m(r^{-1})}{r}$ for a constant $C>0$.
\end{enumerate}}
\end{lemma}

{\everymath{\displaystyle} \begin{proof}
\begin{enumerate}[label=(\roman*)]
	\item $\int_r^\infty \frac{k(\rho)}{\rho} \,d \rho \geq r^2 k(r) \int_r^\infty \frac{\,d\rho}{\rho^3} = \frac{k(r)}{2}$.
	\item $\int_{r/2}^{r} \frac{\omega(\rho) k(\rho)}{\rho} \,d \rho \leq \omega(r) k\left(\frac{r}{2} \right) \int_{r/2}^r \frac{\,d \rho}{\rho} \leq c\log 2 \cdot \omega(r) k(r)$.
	\item $\int_{r}^{\infty} \omega(\rho) \left(-\frac{k'(\rho)}{\rho} + \frac{2 k(\rho)}{\rho^2} \right) \,d \rho \leq \int_{r}^\infty \frac{\omega(\rho)}{\rho^{\epsilon}} \frac{d}{d\rho}\left(-\frac{2(1-\epsilon)^{-1}k(\rho)}{\rho^{1-\epsilon}} \right) \,d\rho \leq \frac{\omega(r)}{r^{\epsilon}} \cdot \frac{2(1-\epsilon)^{-1} k(r)}{r^{1-\epsilon}}$.
	\item $\int_0^r \frac{m(\rho^{-1})^2}{\rho k(\rho)} \,d \rho \leq \int_0^r \frac{(\rho^{1/4} m(\rho^{-1}))^2}{\rho^{3/4} \big(\rho^{3/4} k(\rho) \big)} \,d\rho \leq \frac{(r^{1/4} m(r^{-1}))^2}{r^{3/4} k(r)} \int_0^r \frac{\,d \rho}{\rho^{3/4}} \leq \frac{4 m(r^{-1})^2}{k(r)}$.
	\item $\int_r^\infty \frac{\omega(\rho) m(\rho^{-1})}{\rho^2} \,d \rho \leq r^{-\epsilon} \omega(r) m(r^{-1}) \int_r^\infty \frac{\,d\rho}{\rho^{2-\epsilon}} \leq \frac{(1-\epsilon)^{-1} \omega(r) m(r^{-1})}{r}$. \qedhere
\end{enumerate}
\end{proof}}

\section{Conditional regularity} \label{sec:3}
In this section, we prove that some modulus of continuity for a solution of \eqref{eq:1.2} implies global regularity.

\begin{theorem} \label{thm:8}
Suppose $\theta$ is a bounded weak solution, i.e., $\theta \in L^\infty([0,T];L^p\cap L^\infty)$ where $1\leq p<\infty$, of the slightly supercritical SQG equation
\begin{equation*}
\left\{\begin{array}{rl} \partial_t \theta + (u\cdot \nabla) \theta + \mathcal{L} \theta & = 0, \\ u & = \nabla^\perp \Lambda^{-1} m(\Lambda) \theta, \\ \theta(t=0) &= \theta_0. \end{array}\right.
\end{equation*}
with $k$ and $m$ satisfying the prescribed conditions. Assume that $\Omega$ satisfies \eqref{eq:1.6} with $\omega$ replaced with $\Omega$. Also assume that $k$, $m$, and $\Omega$ satisfy
\[\lim_{r\to 0+} \frac{m(r^{-1}) \Omega(r)}{r k(r)} = 0.\]
If $\theta$ admits a modulus of continuity $\Omega$ for $[0,T]$, then there is a uniform-in-time estimate for $\norm{\nabla \theta(\cdot,t)}_{L^\infty}$ and $\theta$ is a smooth solution on $(0,T]$.
\end{theorem}

\begin{proof}
It suffices to prove the assertion for a smooth solution $\theta$. One may consider regularized equation with the term $-\epsilon \Delta \theta$, get a uniform estimate independent of $\epsilon$, and then take the inviscid limit $\epsilon\to 0+$ to get the desired result. Due to Proposition \ref{prop:4}, it suffices to show that $\norm{\nabla \theta(\cdot,t)}_{L^\infty}$ stays uniformly bounded in time.

\noindent{\it Step 1}. Evolution of $\abs{\nabla \theta}^2$\\
We take gradient on both sides of the equation to get:
\[\partial_t \nabla \theta + u \cdot \nabla^2 \theta + \nabla u \cdot \nabla \theta + \mathcal{L} \nabla \theta = 0.\]
Multiply $\nabla\theta$ to both sides:
\[\nabla\theta\cdot \partial_t \nabla \theta + \nabla \theta \cdot u \cdot \nabla^2 \theta + \nabla \theta \cdot \nabla u \cdot \nabla \theta + \nabla \theta \cdot \mathcal{L} \nabla \theta = 0.\]
Considering singular integral formulation for the nonlocal operator $\mathcal{L}$ gives:
\begin{equation} \label{eq:3.1}
\frac{1}{2} (\partial_t + u \cdot \nabla + \mathcal{L}) \abs{\nabla\theta}^2 + \frac{1}{2}\underbrace{\int_{\R^2} \frac{\abs{\nabla \theta(x,t) - \nabla\theta(y,t)}^2}{\abs{x-y}^{2}} k(x-y) \,d y}_{=: D(x,t)} = -\nabla \theta \cdot \nabla u \cdot \nabla \theta.
\end{equation}
Due to the maximum principle for $\mathcal{L}$, it suffices to show that
\[(\partial_t + u\cdot \nabla + \mathcal{L})\abs{\nabla\theta}^2(x,t) < 0\]
whenever $\abs{\nabla \theta(x,t)}$ is sufficiently large in terms of $\norm{\theta_0}_{L^\infty}, k, m$, and $\Omega$.

\noindent{\it Step 2}. Pointwise lower bound of $D(x,t)$\\
A smooth function $\varphi:[0,\infty) \to \R$ is a non-decreasing cutoff function such that
\[\varphi(x) = 0 \text{ on } x\leq \frac{1}{2},\quad \varphi(x) = 1 \text{ on } x \geq 1,\quad 0\leq\varphi' \leq 4.\]
For some sufficiently small $R=R(x,t)>0$ which will be determined later, we have the following estimate on $D(x,t)$:
\begin{align*}
\frac{D(x,t)}{2\pi} & \geq \frac{1}{2\pi} \int_{\R^2} \frac{\abs{\nabla \theta(x,t) - \nabla \theta(y,t)}^2}{\abs{x-y}^2} k(x-y) \varphi\Big(\frac{\abs{x-y}}{R}\Big) \,d y \\
& \geq \abs{\nabla\theta(x,t)}^2 \frac{1}{2\pi} \int_{\abs{x-y} \geq R} \frac{k(x-y)}{\abs{x-y}^2} \,d y - 2 \nabla \theta(x,t) \cdot \frac{1}{2\pi} \int_{\R^2} \nabla \theta(y,t) \frac{k(x-y)}{\abs{x-y}^2} \varphi\Big(\frac{\abs{x-y}}{R}\Big) \,d y \\
& \geq \abs{\nabla\theta(x,t)}^2 \int_R^\infty \frac{k(\rho)}{\rho} \,d \rho - 2 \nabla \theta (x,t) \cdot \frac{1}{2\pi} \int_{\R^2} (\theta(y,t)-\theta(x,t)) \left\{ \nabla_v \left( \frac{k(\abs{v})}{\abs{v}^2} \varphi \Big(\frac{\abs{v}}{R} \Big) \right)\right\}_{v=x-y}  \,d y \\
& \geq \frac{1}{2} \abs{\nabla\theta(x,t)}^2 k(R) - 2 \abs{ \nabla \theta (x,t)} \int_0^\infty \rho \Omega(\rho) \abs{ \frac{d}{d\rho} \Big(\frac{k(\rho) \varphi(\rho/R)}{\rho^2} \Big)} \,d \rho.
\end{align*}
By Lemma \ref{lemma:7}(ii) and (iii),
\begin{align*}
& \int_0^\infty \rho \Omega(\rho) \abs{ \frac{d}{d\rho} \Big(\frac{k(\rho) \varphi(\rho/R)}{\rho^2} \Big)} \,d \rho \\
& \qquad \leq \frac{4}{R} \int_{R/2}^{R} \frac{\Omega(\rho) k(\rho)}{\rho} \,d\rho + \int_{R/2}^\infty \Omega(\rho) \left(-\frac{k'(\rho)}{\rho} + \frac{2k(\rho)}{\rho^2} \right) \,d\rho \\
& \qquad \leq \frac{c_1}{8} \frac{\Omega(R) k(R)}{R}.
\end{align*}
Therefore,
\[\frac{D(x,t)}{2\pi} \geq \frac{1}{2} \abs{\nabla\theta(x,t)}^2 k(R) - \abs{ \nabla \theta (x,t)} \frac{c_1 \Omega(R) k(R)}{4R} = \frac{1}{2}\abs{\nabla\theta(x,t)}^2 k(R) \left( 1 - \frac{c_1}{2\abs{\nabla\theta(x,t)}} \frac{\Omega(R)}{R} \right).\]
Set $R = R(x,t) >0$ to satisfy
\begin{equation} \label{eq:3.2}
\frac{\Omega(R)}{R} = \frac{\abs{\nabla\theta(x,t)}}{c_1}.
\end{equation}
Then
\begin{equation} \label{eq:3.3}
D(x,t) \geq 4c_2 \abs{\nabla\theta (x,t)}^2 k(R).
\end{equation}
Note that $R(x,t)$ is arbitrarily small when $\abs{\nabla\theta(x,t)}$ is large enough.

\noindent{\it Step 3}. Estimate of $\nabla u$
\[\nabla u(x,t) = \nabla^\perp \Lambda^{-1} m(\Lambda) \nabla \theta(x,t) = P.V. \int_{\R^2} K(x-y) (\nabla \theta(x,t) - \nabla \theta(y,t)) \,d y.\]
We estimate $\nabla u$ by splitting $\R^2$ into two pieces; an inner piece $\abs{x-y} \leq r$ and an outer piece $\abs{x-y} > r$ for some $r = r(x,t) > 0$.

\noindent For the inner piece, we use the Cauchy-Schwartz inequality:
\begin{align*}
\abs{\nabla u_\text{in} (x,t)} & \leq C \sqrt{\bigg(\int_{\abs{x-y} \leq r} \frac{\abs{\nabla\theta(x,t) - \nabla\theta(y,t)}^2}{\abs{x-y}^2} k(x-y) \,d y \bigg)\bigg( \int_{\abs{x-y} \leq r} K(x-y)^2 \frac{\abs{x-y}^2}{k(x-y)} \,d y\bigg)} \\
&\leq C \sqrt{D(x,t) \int_0^r \frac{m(\rho^{-1})^2}{\rho k(\rho)} \,d\rho } \leq c_3 \sqrt{D(x,t) \frac{m(r^{-1})^2}{k(r)}}. \tag{Lemma \ref{lemma:5} and \ref{lemma:7}(iv)}
\end{align*}
For the outer piece, we apply integration by parts and use the modulus of continuity of $\theta$:
\begin{align*}
\abs{\nabla u_\text{out}(x,t)} &\leq \abs{\int_{\abs{x-y} > r} \big(\theta(x,t) - \theta(y,t)\big) \nabla K(x-y) \,d y} \\
& \quad + \abs{\int_{\abs{x-y} = r} \big(\theta(x,t) - \theta(y,t)\big) K(x-y) \nu(y) \,d \sigma(y)} \\
&\leq C \int_{r}^\infty \frac{\Omega(\rho)m(\rho^{-1})}{\rho^2} \,d\rho + C \frac{\Omega(r) m(r^{-1})}{r} \tag{Lemma \ref{lemma:5}} \\
&\leq c_4\frac{\Omega(r) m(r^{-1})}{r}. \tag{Lemma \ref{lemma:7}(v)}
\end{align*}
Therefore,
\begin{align} \label{eq:3.4}
\begin{aligned}
\abs{\nabla u} \abs{\nabla \theta}^2 & \leq \abs{\nabla u_\text{in}} \abs{\nabla \theta}^2 + \abs{\nabla u_\text{out}} \abs{\nabla \theta}^2 \\
& \leq \frac{D}{4} + c_3^2  \frac{m(r^{-1})^2}{k(r)} \abs{\nabla \theta}^4 + c_4 \frac{\Omega(r) m(r^{-1})}{r} \abs{\nabla \theta}^2.
\end{aligned}
\end{align}

\noindent{\it Step 4}. Maximum principle\\
Combining \eqref{eq:3.1}, \eqref{eq:3.3}, and \eqref{eq:3.4},
\begin{align*}
& (\partial_t + u \cdot \nabla + \mathcal{L}) \abs{\nabla\theta}^2 + \frac{D}{2} + 2c_2 k(R) \abs{\nabla\theta}^2 \\
\leq &\, (\partial_t + u \cdot \nabla + \mathcal{L}) \abs{\nabla\theta}^2 + D \\
\leq &\, 2\abs{\nabla u} \abs{\nabla \theta}^2 \leq \frac{D}{2} + 2c_3^2 \frac{m(r^{-1})^2}{k(r)} \abs{\nabla \theta}^4 + 2c_4 \frac{\Omega(r) m(r^{-1})}{r} \abs{\nabla\theta}^2.
\end{align*}
Hence
\[(\partial_t + u \cdot \nabla + \mathcal{L}) \abs{\nabla\theta}^2 + 2c_2 k(R) \abs{\nabla\theta}^2 \leq 2c_3^2 \frac{m(r^{-1})^2}{k(r)} \abs{\nabla \theta}^4 + 2c_4 \frac{\Omega(r) m(r^{-1})}{r} \abs{\nabla\theta}^2.\]
What we want to show are
\begin{equation} \label{eq:3.5}
2c_3^2 \frac{m(r^{-1})^2}{k(r)} \abs{\nabla \theta}^2 < c_2 k(R) ,
\end{equation}
and
\begin{equation} \label{eq:3.6}
2c_4 \frac{\Omega(r) m(r^{-1})}{r} < c_2 k(R) ,
\end{equation}
whenever $\abs{\nabla\theta}$ is sufficiently large. Recall \eqref{eq:3.2} that
\[\abs{\nabla\theta} = c_1 \frac{\Omega(R)}{R}.\]
Hence \eqref{eq:3.5} is equivalent to
\[\frac{m(r^{-1})^2 \Omega(R)^2}{R^2 k(r) k(R)} < \frac{c_2}{2c_1^2 c_3^2}.\]
And \eqref{eq:3.6} is equivalent to
\[\frac{m(r^{-1}) \Omega(r)}{r k(R)} < \frac{c_2}{2c_4}.\]
Set $r = R = R(x,t) > 0$ and then the above two inequalities hold for sufficiently small $R$, since we assumed
\[\lim_{r\to 0+} \frac{m(r^{-1}) \Omega(r)}{r k(r)} = 0. \qedhere\]
\end{proof}

\section{Conservation of modulus of continuity} \label{sec:4}
In this section, we prove that some modulus of continuity of $\theta(\cdot,t)$ is conserved by \eqref{eq:1.2}.

\begin{theorem} \label{thm:9}
Suppose $\theta$ is a bounded weak solution of the slightly supercritical SQG equation
\begin{equation*}
\left\{\begin{array}{rl} \partial_t \theta + (u\cdot \nabla) \theta + \mathcal{L} \theta & = 0, \\ u & = \nabla^\perp \Lambda^{-1} m(\Lambda) \theta, \\ \theta(t=0) &= \theta_0. \end{array}\right.
\end{equation*}
with $k$ and $m$ satisfying the prescribed conditions. Suppose a $C^1$-function $\omega:[0,\infty)\to [0,\infty)$ with $\omega(0)=0$ is nondecreasing, concave, and satisfies \eqref{eq:1.6} and
\[\lim_{r\to 0+} \frac{m(r^{-1})}{k(r)} \frac{\omega'(r)}{\omega(r)} = 0.\]
Then there exists a constant $M>0$ depending on $\norm{\theta_0}_{L^\infty}$, $k$, $m$, and $\omega$ such that if $\theta_0 \in L^p \cap L^\infty$ satisfies the modulus of continuity $\frac{M}{2}\omega$, then $\theta$ satisfies the modulus of continuity $M\omega$ as long as the solution is defined.
\end{theorem}

\begin{proof}
For notational convenience, define $\delta_h f(x) = f(x+h) - f(x)$ for a function $f$. We consider the maximum principle for
\[v(x,h;t) := \left( \frac{\delta_h \theta (x,t)}{\omega(\abs{h})} \right)^2 F(h).\]
The function $F(h) = \exp(-G(|h|))$ is just for the decay of $v$ in $h$, so that $v\in L^p \cap L^\infty(\R^2 \times \R^2)$ for some $1\leq p<\infty$. We set $G$ being smooth, nonnegative, nondecreasing, and $G(r) = 0$ for $0\leq r\leq 1$.

\noindent{\it Step 1}. Evolution of $v$\\
Since
\[\partial_t \theta(x) + u(x) \cdot \nabla \theta(x) + \mathcal{L} \theta(x) = 0,\]
and
\[\partial_t \theta(x+h) + u(x+h) \cdot \nabla \theta(x+h) + \mathcal{L} \theta(x+h) = 0.\]
Subtracting two formulas gives:
\begin{align*}
& \partial_t (\theta(x+h) - \theta(x)) \\
& \qquad + u(x) \cdot \nabla (\theta(x+h) - \theta(x)) + (u(x+h) - u(x)) \cdot \nabla \theta(x+h) + (\mathcal{L}\theta(x+h) - \mathcal{L}\theta(x)) = 0,
\end{align*}
and thus the evolution of $\delta_h \theta$ is given by
\[\big(\partial_t + u \cdot \nabla_x + (\delta_h u) \cdot \nabla_h + \mathcal{L}_x \big) \delta_h \theta = 0.\]
Multiply $\delta_h \theta$ to both sides:
\[\frac{1}{2} \big(\partial_t + u \cdot \nabla_x + (\delta_h u) \cdot \nabla_h + \mathcal{L}_x \big) (\delta_h \theta)^2(x,h,t) + \frac{1}{2} \underbrace{\int_{\R^2} \frac{\abs{\delta_h \theta(x,t) - \delta_h \theta(y,t)}^2}{\abs{x-y}^2} k(x-y) \,d y }_{=: D_h (x,t)} = 0.\]
Therefore
\begin{equation} \label{eq:4.1}
\big(\partial_t + u \cdot \nabla_x + (\delta_h u) \cdot \nabla_h + \mathcal{L}_x \big) v(x,h,t) + D_h(x,t) \frac{F(h)}{\omega(\abs{h})^{2}} = - (\delta_h u) \cdot \frac{h}{\abs{h}} \left( G'(\abs{h}) + \frac{2 \omega'(\abs{h})}{\omega(\abs{h})} \right) v.
\end{equation}

\noindent{\it Step 2}. Breakthrough moment and maximum principle\\
Fix $M>0$, which will be determined later. Suppose $\theta_0$ satisfies a modulus of continuity $\frac{M}{2} \omega$ and $\theta(t)$ admits a modulus of continuity $M\omega$ for $0\leq t\leq t^\ast$. However, at the time $t = t^\ast$,
\begin{equation} \label{eq:4.2}
\abs{\delta_h \theta(x,t^\ast)} \geq M \omega(\abs{h})
\end{equation}
for some $x,h$. Suppose we prove that
\begin{equation} \label{eq:4.3}
(\partial_t + u\cdot \nabla_x + (\delta_h u) \cdot \nabla_h + \mathcal{L}_x)v(x,h,t^\ast) < 0,
\end{equation}
whenever the such event occurs, then we get contraction from the maximum principle of $\mathcal{L}_x$. Hence such an event is impossible, and the modulus of continuity $M\omega(\cdot)$ is conserved.

Therefore it suffices to show that
\[(\partial_t + u\cdot \nabla_x + (\delta_h u) \cdot \nabla_h + \mathcal{L}_x)v(x,h,t^*) < 0,\]
whenever
\[M \omega(\abs{h}) \leq \abs{\delta_h \theta(x,t^*)} \leq 2M \omega(\abs{h}).\]
If \eqref{eq:4.2} is satisfed, 
\[M\omega(\abs{h}) \leq \abs{\delta_h \theta(x,t^*)} \leq 2\norm{\theta_0}_{L^\infty}.\]
Hence, we only need to consider $h$ with $M \omega(\abs{h}) \leq 2 \norm{\theta_0}_{L^\infty}$. Let $r_0 > 0$ be $M\omega(r_0) = 2\norm{\theta_0}_{L^\infty}$. We choose $M$ large so that $r_0 < 1$ and thus $F(h) = 1$, $G'(\abs{h})=0$ for $\abs{h} \leq r_0$.
Due to \eqref{eq:4.1}, it suffices to show that if $M$ is large enough, then
\begin{equation} \label{eq:4.4}
D_h(x,t^*) \geq 2 \frac{\omega'(\abs{h})}{\omega(\abs{h})} \abs{\delta_h u(x,t^*)} \abs{\delta_h \theta(x,t^*)}^2,
\end{equation}
whenever $M \omega(\abs{h}) \leq \abs{\delta_h \theta(x,t^*)} \leq 2M \omega(\abs{h})$. Let us denote $t^*$ by simply $t$ from here.

\noindent{\it Step 3}. Pointwise lower bound of $D_h$\\
A smooth function $\varphi:[0,\infty) \to \R$ is a non-decreasing cutoff function such that
\[\varphi(r) = 0 \text{ on } r\leq \frac{1}{2},\quad \varphi(r) = 1 \text{ on } r \geq 1,\quad 0\leq\varphi' \leq 4.\]
For some sufficiently small $R=R(x,h,t)\geq 6\abs{h}$ which will be determined later, we have the following estimate on $D_h(x,t)$:
\begin{align*}
\frac{D_h(x,t)}{2\pi} & \geq \frac{1}{2\pi} \int_{\R^2} \frac{\abs{\delta_h \theta(x,t) - \delta_h \theta(y,t)}^2}{\abs{x-y}^2} k(x-y) \varphi\Big(\frac{\abs{x-y}}{R}\Big) \,d y \\
& \geq \abs{\delta_h \theta(x,t)}^2 \frac{1}{2\pi} \int_{\abs{x-y} \geq R} \frac{k(x-y)}{\abs{x-y}^2} \,d y - 2 \delta_h \theta(x,t) \cdot \frac{1}{2\pi} \int_{\R^2} \delta_h \theta(y,t) \frac{k(x-y)}{\abs{x-y}^2} \varphi\Big(\frac{\abs{x-y}}{R}\Big) \,d y \\
& \geq \abs{\delta_h \theta(x,t)}^2 \int_R^\infty \frac{k(\rho)}{\rho} \,d \rho \\
& \qquad - 2 \delta_h \theta (x,t) \cdot \frac{1}{2\pi} \int_{\R^2} (\theta(y,t)- \theta(x,t)) \bigg\{ \delta_h \left( \frac{k(\abs{v})}{\abs{v}^2} \varphi \Big(\frac{\abs{v}}{R} \Big) \right)\bigg\}_{v=x-y} \,d y.
\end{align*}
By the mean value theorem,
\begin{align*}
& \abs{\delta_h \left( \frac{k(\abs{v})}{\abs{v}^2} \varphi \Big(\frac{\abs{v}}{R} \Big) \right)}_{v=x-y} \leq \abs{h} \abs{\nabla_v \left( \frac{k(\abs{v})}{\abs{v}^2} \varphi \Big(\frac{\abs{v}}{R} \Big) \right)}_{v=x-y+h^\ast} \\
& \hspace{2cm} \leq \abs{h} \left\{ \frac{4}{R} \frac{k(\rho^\ast)}{(\rho^\ast)^2} \mathbf{1}_{R\geq \rho^\ast\geq \frac{R}{2}} + \left(- \frac{k'(\rho^\ast)}{(\rho^\ast)^2} + \frac{2k(\rho^\ast)}{(\rho^\ast)^3} \right) \mathbf{1}_{\rho^\ast \geq \frac{R}{2}} \right\}_{\rho^\ast = \abs{x-y+h^\ast}}
\end{align*}
Since $\rho^\ast = \abs{x-y+h^\ast} \geq \frac{R}{2}$ and $\abs{h^\ast} \leq \abs{h} \leq \frac{R}{6}$, it follows $\rho = \abs{x-y} \geq \frac{R}{3}$ and $\rho^\ast \geq \rho - \frac{R}{6} \geq \frac{\rho}{2}$. Hence the above formula is less than equal to the same formula where $\rho^\ast$ replaced by $\rho/2$, $R \geq \rho^* \geq \frac{R}{2}$ replaced by $\frac{7R}{6} \geq \rho \geq \frac{R}{3}$, and $\rho^\ast\geq \frac{R}{2}$ replaced by $\rho \geq \frac{R}{3}$.
In addition, regarding the modulus of continuity of $\theta(\cdot,t)$, we have
\[\abs{\theta(y,t)-\theta(x,t)} \leq M \omega(\abs{x-y}).\]
Therefore, wrapping up the estimates with Lemma \ref{lemma:7}, we have the following estimate for $D_h$ with a constant $c_5 (> 2)$.
\begin{align*}
\frac{D_h(x,t)}{2\pi} & \geq \frac{1}{2} \abs{\delta_h\theta(x,t)}^2 k(R) - 2 \abs{\delta_h \theta(x,t)} \abs{h} M \cdot c_5 \frac{\omega(R) k(R)}{R} \\
& \geq \frac{1}{2} \abs{\delta_h\theta(x,t)}^2 k(R) \left( 1 - \frac{4 c_5 \abs{h} M}{\abs{\delta_h \theta(x,t)}} \frac{\omega(R)}{R} \right).
\end{align*}
Set $R = R(x,h,t) > 0$ to satisfy
\[\frac{\omega(R)}{R} = \frac{1}{8c_5 M} \frac{\abs{\delta_h \theta(x,t)}}{\abs{h}}.\]
Since
\[\frac{\omega(R)}{R} = \frac{1}{8c_5 M} \frac{\abs{\delta_h \theta(x,t)}}{\abs{h}} \leq \frac{1}{4c_5} \frac{\omega(\abs{h})}{\abs{h}} < \frac{\omega(6\abs{h})}{6\abs{h}},\]
we get $R \geq 6\abs{h}$ as assumed. On the other hand,
\[\frac{\omega(R)}{R} = \frac{1}{8c_5 M} \frac{\abs{\delta_h \theta(x,t)}}{\abs{h}} \geq \frac{1}{8c_5} \frac{\omega(\abs{h})}{\abs{h}},\]
and thus
\[8c_5 \frac{\abs{h}^{1-\epsilon}}{R^{1-\epsilon}} \geq \frac{\omega(\abs{h}) \abs{h}^{-\epsilon}}{\omega(R) R^{-\epsilon}} \geq 1.\]
Hence $R \leq (8c_5)^{1/(1-\epsilon)} \abs{h}$. Summing up,
\begin{equation} \label{eq:4.5}
D_h(x,t) \geq c k(R) \abs{\delta_h \theta(x,t)}^2 \geq 8c_6 k(\abs{h}) \abs{\delta_h \theta(x,t)}^2.\end{equation}

\noindent{\it Step 4}. Estimate of $\delta_h u$
\[\delta_h u(x,t) = P.V. \int_{\R^2} K(x-y) (\delta_h \theta(x,t) - \delta_h \theta(y,t)) \,d y.\]
We estimate $\nabla u$ by splitting $\R^2$ into two pieces; an inner piece $\abs{x-y} \leq 3\abs{h}$ and an outer piece $\abs{x-y} > 3\abs{h}$.

\noindent For the inner piece, we use the Cauchy-Schwartz inequality:
\begin{equation} \label{eq:4.6}
\abs{\delta_h u_{\text{in}}(x,t)} \leq c_7 \sqrt{D_h(x,t) \frac{m(\abs{h}^{-1})^2}{k(\abs{h})}}.
\end{equation}
For the outer piece:
\begin{align*}
\abs{\delta_h u_{\text{out}}(x,t)} &= \abs{\int_{\abs{x-y} > 3\abs{h}} K(x-y)(\theta(y,t) - \theta(y+h,t)) \,d y } \\
&\leq \abs{\int_{\abs{x-y} > 3\abs{h}} (K(x-y) - K(x-y+h))\theta(y,t) \,d y } \\
&\qquad + \abs{\int_{\substack{\{y:\abs{x-y} > 3\abs{h}\} \triangle \\ \{y:\abs{x-y+h} > 3\abs{h}\} }} K(x-y+h) \theta(y,t) \,d y } \\
\end{align*}
The first term is estimated as follows:
\begin{align*}
&\leq \norm{\theta_0}_{L^\infty} \int_{\abs{x-y} > 3\abs{h}} \abs{h} \abs{\nabla K(x-y+h^\ast)} \,d y \tag{$\abs{h^\ast} \leq \abs{h} < \frac{1}{3}\abs{x-y}$} \\
&\leq c \norm{\theta_0}_{L^\infty} \abs{h} \int_{\abs{x-y} > 3\abs{h}} \abs{\frac{2}{3}(x-y)}^{-3} m\big(\abs{{\textstyle \frac{2}{3}}(x-y)}^{-1} \big) \,d y \\
&\leq c \norm{\theta_0}_{L^\infty} \abs{h} \int_{\abs{h}}^\infty \frac{m\big(\rho^{-1}\big)}{\rho^2} \,d \rho \leq c_3 \norm{\theta_0}_{L^\infty} m\big( \abs{h}^{-1} \big).
\end{align*}
And the second term is estimated as follows:
\begin{align*}
&\leq \norm{\theta_0}_{L^\infty} \int_{2\abs{h} < \abs{x-y+h} \leq 4\abs{h}} \hspace{-1.5cm} \abs{K(x-y+h)} \,d y \leq C \norm{\theta_0}_{L^\infty} \int_{2\abs{h}}^{4\abs{h}} \frac{m(\rho^{-1})}{\rho} \,d \rho \leq C \norm{\theta_0}_{L^\infty} m(\abs{h}^{-1}).
\end{align*}
Therefore,
\begin{equation} \label{eq:4.7}
\abs{\delta_h u_{\text{out}}(x,t)} \leq c_8 \norm{\theta_0}_{L^\infty} m(\abs{h}^{-1}).
\end{equation}

\noindent{\it Step 5}. Proof of \eqref{eq:4.4}\\
By the estimates \eqref{eq:4.5}, \eqref{eq:4.6}, and \eqref{eq:4.7}, it is enough to show that
\begin{equation} \label{eq:4.8}
\frac{D_h}{4} + 2c_6 k(\abs{h}) \abs{\delta_h \theta}^2 \geq \left(c_7 \sqrt{D_h(x,t) \frac{m(\abs{h}^{-1})^2}{k(\abs{h})}} + c_8 \norm{\theta_0}_{L^\infty} m(\abs{h}^{-1}) \right) \frac{\omega'(\abs{h})}{\omega(\abs{h})} \abs{\delta_h \theta}^2.
\end{equation}
By Young's inequality,
\begin{align*}
\text{R.H.S. of \eqref{eq:4.8}} &\leq \frac{D_h}{4} + c_7^2 \frac{m(\abs{h}^{-1})^2}{k(\abs{h})} \left(\frac{\omega'(\abs{h})}{\omega(\abs{h})}\right)^2 \abs{\delta_h \theta}^4 + c_8 \norm{\theta_0}_{L^\infty} m(\abs{h}^{-1}) \left(\frac{\omega'(\abs{h})}{\omega(\abs{h})}\right) \abs{\delta_h \theta}^2.
\end{align*}
Our goal \eqref{eq:4.3} is verified if
\begin{equation} \label{eq:4.9}
c_7^2 \frac{m(\abs{h}^{-1})^2}{k(\abs{h})} \left(\frac{\omega'(\abs{h})}{\omega(\abs{h})}\right)^2 \abs{\delta_h \theta}^2 \leq c_6 k(\abs{h}),
\end{equation}
and
\begin{equation} \label{eq:4.10}
c_8 \norm{\theta_0}_{L^\infty} m(\abs{h}^{-1}) \left(\frac{\omega'(\abs{h})}{\omega(\abs{h})}\right) \leq c_6 k(\abs{h}).
\end{equation}
Since $\abs{\delta_h \theta} \leq 2\norm{\theta_0}_{L^\infty}$, it is enough to show
\[\left( \frac{m(\abs{h}^{-1})}{k(\abs{h})} \frac{\omega'(\abs{h})}{\omega(\abs{h})}\right)^2 \leq \frac{c_6}{4 c_7^2 \norm{\theta_0}_{L^\infty}^2},\]
instead of \eqref{eq:4.9}. And \eqref{eq:4.10} is equivalent to
\begin{equation} \label{eq:4.11}
\frac{m(\abs{h}^{-1})}{k(\abs{h})} \frac{\omega'(\abs{h})}{\omega(\abs{h})} \leq \frac{c_6}{c_8 \norm{\theta_0}_{L^\infty}}.
\end{equation}
The above two inequalities hold for sufficiently small $\abs{h}$ by our assumption on $k, m,$ and $\omega$. It completes the proof that a modulus of continuity $M\omega$ is conserved for sufficiently large $M$.
\end{proof}

\noindent\textbf{Main Theorem 1.} Assume that $\theta_0 \in \mathcal{S}(\R^2)$. Suppose there is a non-decreasing, continuous, concave function $\omega:[0,\infty)\to[0,\infty)$ satisfying $\omega(0)=0$, \eqref{eq:1.4}, \eqref{eq:1.5}, and for some $\epsilon>0$,
\begin{equation*}
\lim_{r\to 0+} \frac{\omega(r)}{r^{\epsilon}} = \infty, \quad \text{and}\quad \lim_{r\to \infty} \frac{\omega(r)}{r^\epsilon} = 0.
\end{equation*}
Then there exists a global smooth solution $\theta$ of the slightly supercritical SQG equation \eqref{eq:1.2}.

\begin{proof}
Fix arbitrary $T>0$, then there exists a bounded weak solution on time $[0,T]$. The initial $\theta_0\in \mathcal{S}(\R^2)$ automatically satisfies the modulus of continuity $\frac{M}{2} \omega$ for sufficiently large $M>0$. Since \eqref{eq:1.5} is satisfied, Theorem \ref{thm:9} implies that $\theta(\cdot,t)$ admits the modulus of continuity $M\omega$ for $t\in [0,T]$. Since \eqref{eq:1.4} is satisfied, by Theorem \ref{thm:8} we have that $\theta$ is smooth. It is clear from the proof that an initial $\theta_0 \in W^{1,\infty}\cap H^{s}(\R^2)$ with $s>1$ is enough to guarantee a smooth solution. To be precise, an initial being Lipschitz ensures there is $M>0$ which $\frac{M}{2} \omega$ is a modulus of continuity of the initial. And $H^s$, $s>1$ regularity is required to show local well-posedness. Lastly, note that it is known that a classical solution instantaneous becomes smooth, for instance see \cite{Don09a}.
\end{proof}

\begin{proof}[Proof of Corollary \ref{cor:1}] Due to Lemma \ref{lemma:5} and \ref{lemma:6}, the equation \eqref{eq:1.7} corresponds to the case
\[k(r) \sim \frac{1}{r(-\log r)^{\alpha_1}}~~\text{for } r\ll1, \quad \text{and}\quad m(\zeta) \sim (\log \abs{\zeta})^{\alpha_2}~~\text{for } \abs{\zeta} \gg 1.\]
Set a function $\omega$ as
\[\omega(r) = \frac{1}{(-\log r)^\beta}, \quad 0<r<1.\]
Then it satisfies the condition
\[\frac{m(r^{-1})}{k(r)} \frac{\omega'(r)}{\omega(r)} \sim \frac{r(-\log r)^{\alpha_1 + \alpha_2}}{r(-\log r)} \to 0\quad\text{as } r\to 0+,\]
if $\alpha_1 + \alpha_2 < 1$. Due to Theorem \ref{thm:9}, there exists $M>0$ so that $\theta(\cdot,t)$ admits
\[\Omega(r) = \min\{M \omega(r), 2\norm{\theta_0}_{L^\infty}\}\]
as a modulus of continuity. Take our $\beta$ to be $\beta > \alpha_1 + \alpha_2$ and then
\[\frac{m(r^{-1}) \Omega(r)}{rk(r)} \sim \Omega(r) (-\log r)^{\alpha_1 + \alpha_2} \to 0,\quad\text{as } r \to 0+.\]
According to Theorem \ref{thm:8}, we obtain that the logarithmically supercritical SQG equation \eqref{eq:1.7} admits a global smooth solution if $\alpha_1 + \alpha_2 < 1$.
\end{proof}

\section{Uniform-in-time Gradient Estimate} \label{sec:5}
In this section, we prove the estimate \eqref{eq:1.9} for the critical SQG equation \eqref{eq:1.8}. Note that the critical SQG equation is a special case in the previous discussion that
\[k(r) = \frac{1}{2\pi r}, \quad m(\zeta) = 1.\]
Main Theorem 2 can be seen as consequence of section 3 and 4. We repeat discussion on section 3 and 4, in the sprit of the critical SQG, which gives a self-contained proof for Main Theorem 2.

\noindent\textbf{Main Theorem 2.}
The critical SQG equation
\begin{equation} \label{eq:A1.7}
\left\{\begin{array}{rl} \partial_t \theta + (u\cdot \nabla) \theta + \Lambda \theta & = 0,\\
u & = \nabla^\perp \Lambda^{-1} \theta,\\
\theta(t=0) &= \theta_0 \end{array}\right.
\end{equation}
with initial $\theta_0 \in \mathcal{S}(\R^2)$ has a unique global smooth solution. In addition, for any $\gamma>1$,
\begin{equation} \label{eq:A1.8}
\norm{\nabla \theta}_{L^\infty_{t,x}} \leq C \norm{\nabla \theta_0}_{L^\infty} \exp(C \norm{\theta_0}_{L^\infty}^{\gamma}),
\end{equation}
where constants depend only on $\gamma$.

A solution $\theta$ to the critical SQG equation exhibits the following rescaling property:
\[\theta_\lambda(x,t) = \theta(\lambda x, \lambda t).\]
Since this rescaling changes $\norm{\nabla\theta}_{L^\infty}$ by $\lambda$ times and remains $\norm{\theta}_{L^\infty}$ the same, we may assume that $\norm{\nabla \theta_0}_{L^\infty} = 1$. In addition, it suffices to show our estimate for $\norm{\theta_0}_{L^\infty} \hspace{-.3em} \geq 1$ as we already have a double exponential estimate.

Let $\omega(r) = (-\log r)^{-\beta}$ for some $\beta > 0$ and let $\Omega(r) = \min\{M \omega(r), 2\norm{\theta_0}_{L^\infty}\}$ with a constant $M$ which will be determined later. We split the proof into two parts:

\begin{proposition} \label{prop:10}
Suppose $\theta$ is a bounded weak solution, i.e., $\theta \in L^\infty([0,T];L^p \cap L^\infty)$ where $1\leq p<\infty$, of the critical SQG equation \eqref{eq:A1.7}. Then there exists constant $M = C_\beta \norm{\theta_0}_{L^\infty}^{1+\beta}$ so that an initial $\theta(\cdot,t)$ admits modulus of continuity $\Omega$ as long as the solution is defined, whenever an initial $\theta_0$ admits $\min\{\frac{M}{2} \omega(r), 2\norm{\theta_0}_{L^\infty}\}$.
\end{proposition}

\begin{proposition} \label{prop:11}
Suppose $\theta$ is a bounded weak solution of the critical SQG equation \eqref{eq:A1.7}. If $\theta(\cdot,t)$ admits a modulus of continuity $\Omega$ for $t\in [0,T]$, then
\[\norm{\nabla \theta}_{L^\infty_{t,x}} \leq C \norm{\nabla \theta_0}_{L^\infty} \exp(C M^{\frac{1}{\beta}}),\]
where constants depend only on $\beta$. Thus $\theta$ is a smooth solution on $(0,T]$.
\end{proposition}

\begin{proof}[Proof of Proposition \ref{prop:10}]
Let $r_0 \in (0,1)$ be $M\omega(r_0) = 2\norm{\theta}_{L^\infty}$. Since $L^\infty$-norm is conserved, it suffices to show the modulus of continuity for distance less than $r_0$.

For notational convenience, define $\delta_h f(x) = f(x+h) - f(x)$ for a function $f$. We consider the maximum principle for
\[v(x,h;t) := \left( \frac{\delta_h \theta (x,t)}{\omega(\abs{h})} \right)^2 F(h).\]
The function $F(h) = \exp(-G(|h|))$ is just for the decay of $v$ in $h$, so that $v\in L^p \cap L^\infty(\R^2 \times \R^2)$ for some $1\leq p<\infty$. We set $G$ being smooth, nonnegative, nondecreasing, and $G(r) = 0$ for $0\leq r\leq 1$.

\noindent{\it Step 1}. Evolution of $v$\\
Since
\[\partial_t \theta(x) + u(x) \cdot \nabla \theta(x) + \Lambda \theta(x) = 0,\]
and
\[\partial_t \theta(x+h) + u(x+h) \cdot \nabla \theta(x+h) + \Lambda \theta(x+h) = 0.\]
Subtracting two formulas gives:
\begin{align*}
& \partial_t (\theta(x+h) - \theta(x)) \\
& \qquad + u(x) \cdot \nabla (\theta(x+h) - \theta(x)) + (u(x+h) - u(x)) \cdot \nabla \theta(x+h) + (\Lambda \theta(x+h) - \Lambda \theta(x)) = 0,
\end{align*}
and thus the evolution of $\delta_h \theta$ is given by
\[\big(\partial_t + u \cdot \nabla_x + (\delta_h u) \cdot \nabla_h + \Lambda_x \big) \delta_h \theta = 0.\]
Multiply $\delta_h \theta$ to both sides:
\[\frac{1}{2} \big(\partial_t + u \cdot \nabla_x + (\delta_h u) \cdot \nabla_h + \Lambda_x \big) (\delta_h \theta)^2(x,h,t) + \frac{1}{2} \underbrace{ \frac{1}{2\pi} \int_{\R^2} \frac{\abs{\delta_h \theta(x,t) - \delta_h \theta(y,t)}^2}{\abs{x-y}^3} \,d y }_{=: D_h (x,t)} = 0.\]
Therefore
\begin{equation} \label{eq:A4.1}
\big(\partial_t + u \cdot \nabla_x + (\delta_h u) \cdot \nabla_h + \Lambda_x \big) v(x,h,t) + D_h(x,t) \frac{F(h)}{\omega(\abs{h})^{2}} = - (\delta_h u) \cdot \frac{h}{\abs{h}} \left( G'(\abs{h}) + \frac{2 \omega'(\abs{h})}{\omega(\abs{h})} \right) v.
\end{equation}

\noindent{\it Step 2}. Breakthrough moment and maximum principle\\
Fix $M>0$, which will be determined later. Suppose $\theta_0$ satisfies a modulus of continuity $\frac{M}{2} \omega$ and $\theta(t)$ admits a modulus of continuity $M\omega$ for $0\leq t\leq t^\ast$. However, at the time $t = t^\ast$,
\begin{equation} \label{eq:A4.2}
\abs{\delta_h \theta(x,t^\ast)} \geq M \omega(\abs{h})
\end{equation}
for some $x,h$. Suppose we prove that
\begin{equation} \label{eq:A4.3}
(\partial_t + u\cdot \nabla_x + (\delta_h u) \cdot \nabla_h + \Lambda_x)v(x,h,t^\ast) < 0,
\end{equation}
whenever the such event occurs, then we get contraction from the maximum principle of $\Lambda_x$. Hence such an event is impossible, and the modulus of continuity $M\omega(\cdot)$ is conserved.

Therefore it suffices to show that
\[(\partial_t + u\cdot \nabla_x + (\delta_h u) \cdot \nabla_h + \Lambda_x)v(x,h,t^*) < 0,\]
whenever
\[M \omega(\abs{h}) \leq \abs{\delta_h \theta(x,t^*)} \leq 2M \omega(\abs{h}).\]
If \eqref{eq:A4.2} is satisfed, 
\[M\omega(\abs{h}) \leq \abs{\delta_h \theta(x,t^*)} \leq 2\norm{\theta_0}_{L^\infty}.\]
Hence, we only need to consider $\abs{h} \leq r_0$.
Due to \eqref{eq:A4.1}, it suffices to show that
\begin{equation} \label{eq:A4.4}
\frac{D_h(x,t^*)}{2\beta} \geq \frac{1}{\abs{h}(-\log\abs{h})} \abs{\delta_h u(x,t^*)} \abs{\delta_h \theta(x,t^*)}^2,
\end{equation}
whenever $M \omega(\abs{h}) \leq \abs{\delta_h \theta(x,t^*)} \leq 2M \omega(\abs{h})$, $\abs{h} \leq r_0$. Let us denote $t^*$ by simply $t$ from here.

\noindent{\it Step 3}. Pointwise lower bound of $D_h$\\
A smooth function $\varphi:[0,\infty) \to \R$ is a non-decreasing cutoff function such that
\[\varphi(r) = 0 \text{ on } r\leq \frac{1}{2},\quad \varphi(r) = 1 \text{ on } r \geq 1,\quad 0\leq\varphi' \leq 4.\]
For some sufficiently small $R=R(x,h,t)\geq 6\abs{h}$ which will be determined later, we have the following estimate on $D_h(x,t)$:
\begin{align*}
D_h(x,t) & \geq \frac{1}{2\pi} \int_{\R^2} \frac{\abs{\delta_h \theta(x,t) - \delta_h \theta(y,t)}^2}{\abs{x-y}^3} \varphi\Big(\frac{\abs{x-y}}{R}\Big) \,d y \\
& \geq \abs{\delta_h \theta(x,t)}^2 \frac{1}{2\pi} \int_{\abs{x-y} \geq R} \frac{1}{\abs{x-y}^3} \,d y - 2 \delta_h \theta(x,t) \cdot \frac{1}{2\pi} \int_{\R^2} \delta_h \theta(y,t) \frac{1}{\abs{x-y}^3} \varphi\Big(\frac{\abs{x-y}}{R}\Big) \,d y \\
& \geq \abs{\delta_h \theta(x,t)}^2 \int_R^\infty \frac{1}{\rho^2} \,d \rho \\
& \qquad - 2 \delta_h \theta (x,t) \cdot \frac{1}{2\pi} \int_{\R^2} (\theta(y,t)- \theta(x,t)) \bigg\{ \delta_h \left( \frac{1}{\abs{v}^3} \varphi \Big(\frac{\abs{v}}{R} \Big) \right)\bigg\}_{v=x-y} \,d y.
\end{align*}
By the mean value theorem,
\begin{align*}
& \abs{\delta_h \left( \frac{1}{\abs{v}^3} \varphi \Big(\frac{\abs{v}}{R} \Big) \right)}_{v=x-y} \leq \abs{h} \abs{\nabla_v \left( \frac{1}{\abs{v}^3} \varphi \Big(\frac{\abs{v}}{R} \Big) \right)}_{v=x-y+h^\ast} \\
& \hspace{2cm} \leq \abs{h} \left\{ \frac{4}{R} \frac{1}{(\rho^\ast)^3} \mathbf{1}_{R\geq \rho^\ast\geq \frac{R}{2}} + \frac{3}{(\rho^*)^4} \mathbf{1}_{\rho^\ast \geq \frac{R}{2}} \right\}_{\rho^\ast = \abs{x-y+h^\ast}}
\end{align*}
Since $\rho^\ast = \abs{x-y+h^\ast} \geq \frac{R}{2}$ and $\abs{h^\ast} \leq \abs{h} \leq \frac{R}{6}$, it follows $\rho = \abs{x-y} \geq \frac{R}{3}$ and $\rho^\ast \geq \rho - \frac{R}{6} \geq \frac{\rho}{2}$. Hence the above formula is less than equal to the same formula where $\rho^\ast$ replaced by $\rho/2$, $R \geq \rho^* \geq \frac{R}{2}$ replaced by $\frac{7R}{6} \geq \rho \geq \frac{R}{3}$, and $\rho^\ast\geq \frac{R}{2}$ replaced by $\rho \geq \frac{R}{3}$.
In addition, regarding the modulus of continuity of $\theta(\cdot,t)$, we have
\[\abs{\theta(y,t)-\theta(x,t)} \leq M \omega(\abs{x-y}).\]
Therefore, we have the following estimate for $D_h$ with a constant $c_5 (> 3)$.
\begin{align*}
D_h(x,t)& \geq \frac{\abs{\delta_h\theta(x,t)}^2}{R} - 2 \abs{\delta_h \theta(x,t)} \abs{h} M \cdot c_5 \frac{\omega(R)}{R^2} \\
& \geq \frac{\abs{\delta_h\theta(x,t)}^2}{R} \left( 1 - \frac{2 c_5 \abs{h} M}{\abs{\delta_h \theta(x,t)}} \frac{\omega(R)}{R} \right).
\end{align*}
Set $R = R(x,h,t) > 0$ to satisfy
\[\frac{\omega(R)}{R} = \frac{1}{4c_5 M} \frac{\abs{\delta_h \theta(x,t)}}{\abs{h}}.\]
Since
\[\frac{\omega(R)}{R} = \frac{1}{4c_5 M} \frac{\abs{\delta_h \theta(x,t)}}{\abs{h}} \leq \frac{1}{2c_5} \frac{\omega(\abs{h})}{\abs{h}} < \frac{\omega(6\abs{h})}{6\abs{h}},\]
we get $R \geq 6\abs{h}$ as assumed. On the other hand,
\[\frac{\omega(R)}{R} = \frac{1}{4c_5 M} \frac{\abs{\delta_h \theta(x,t)}}{\abs{h}} \geq \frac{1}{4c_5} \frac{\omega(\abs{h})}{\abs{h}},\]
and thus
\[4c_5 \frac{\abs{h}^{1-\gamma}}{R^{1-\gamma}} \geq \frac{\omega(\abs{h}) \abs{h}^{-\gamma}}{\omega(R) R^{-\gamma}} \geq 1.\]
Hence $R \leq (4c_5)^{1/(1-\gamma)} \abs{h}$. Summing up,
\begin{equation} \label{eq:A4.5}
D_h(x,t) \geq \frac{\abs{\delta_h \theta(x,t)}^2}{2R} \geq 8c_6 \frac{\abs{\delta_h \theta(x,t)}^2}{\abs{h}}.\end{equation}

\noindent{\it Step 4}. Estimate of $\delta_h u$
\[\delta_h u(x,t) = P.V. \int_{\R^2} \frac{(x-y)^\perp}{\abs{x-y}^3} (\delta_h \theta(x,t) - \delta_h \theta(y,t)) \,d y.\]
Name the kernel by $K(z) = z^\perp/\abs{z}^3$. We estimate $\nabla u$ by splitting $\R^2$ into two pieces; an inner piece $\abs{x-y} \leq 3\abs{h}$ and an outer piece $\abs{x-y} > 3\abs{h}$.

\noindent For the inner piece, we use the Cauchy-Schwartz inequality:
\begin{align} \label{eq:A4.6} \begin{split}
\abs{\delta_h u_{\text{in}}(x,t)} & \leq C \sqrt{\bigg(\int_{\abs{x-y} \leq 3\abs{h}} \frac{\abs{\delta_h\theta(x,t) - \delta_h\theta(y,t)}^2}{\abs{x-y}^3} \,d y \bigg)\bigg( \int_{\abs{x-y} \leq 3\abs{h}} \frac{1}{\abs{x-y}} \,d y\bigg)} \\
&= c_7 \sqrt{D_h(x,t) \abs{h}}.
\end{split}\end{align}
For the outer piece:
\begin{align*}
\abs{\delta_h u_{\text{out}}(x,t)} &= \abs{\int_{\abs{x-y} > 3\abs{h}} K(x-y)(\theta(y,t) - \theta(y+h,t)) \,d y } \\
&\leq \abs{\int_{\abs{x-y} > 3\abs{h}} (K(x-y) - K(x-y+h))\theta(y,t) \,d y } \\
&\qquad + \abs{\int_{\substack{\{y:\abs{x-y} > 3\abs{h}\} \triangle \\ \{y:\abs{x-y+h} > 3\abs{h}\} }} K(x-y+h) \theta(y,t) \,d y } \\
\end{align*}
The first term is estimated as follows:
\begin{align*}
&\leq \norm{\theta_0}_{L^\infty} \int_{\abs{x-y} > 3\abs{h}} \abs{h} \abs{\nabla K(x-y+h^\ast)} \,d y \tag{$\abs{h^\ast} \leq \abs{h} < \frac{1}{3}\abs{x-y}$} \\
&\leq c \norm{\theta_0}_{L^\infty} \abs{h} \int_{\abs{x-y} > 3\abs{h}} \abs{x-y}^{-3} \,d y = C \norm{\theta_0}_{L^\infty}.
\end{align*}
And the second term is estimated as follows:
\begin{align*}
&\leq \norm{\theta_0}_{L^\infty} \int_{2\abs{h} < \abs{x-y+h} \leq 4\abs{h}} \hspace{-1.5cm} \abs{K(x-y+h)} \,d y \leq C \norm{\theta_0}_{L^\infty} \int_{2\abs{h}}^{4\abs{h}} \frac{1}{\rho} \,d \rho \leq C \norm{\theta_0}_{L^\infty}.
\end{align*}
Therefore,
\begin{equation} \label{eq:A4.7}
\abs{\delta_h u_{\text{out}}(x,t)} \leq c_8 \norm{\theta_0}_{L^\infty}.
\end{equation}
\noindent{\it Step 5}. Proof of \eqref{eq:A4.4}\\
By the estimates \eqref{eq:A4.5}, \eqref{eq:A4.6}, and \eqref{eq:A4.7}, it is enough to show that
\begin{equation} \label{eq:A4.8}
\frac{D_h}{4\beta} + \frac{2c_6}{\beta} \frac{\abs{\delta_h \theta}^2}{\abs{h}} \geq \left(c_7 \sqrt{D_h(x,t) \abs{h}} + c_8 \norm{\theta_0}_{L^\infty} \right) \frac{1}{\abs{h}(-\log\abs{h})} \abs{\delta_h \theta}^2.
\end{equation}
By Young's inequality,
\begin{align*}
\text{R.H.S. of \eqref{eq:A4.8}} &\leq \frac{D_h}{4\beta} + \beta c_7^2 \abs{h} \left(\frac{1}{\abs{h}(-\log\abs{h})}\right)^2 \abs{\delta_h \theta}^4 + c_8 \norm{\theta_0}_{L^\infty} \left(\frac{1}{\abs{h}(-\log\abs{h})}\right) \abs{\delta_h \theta}^2.
\end{align*}
If we verify that
\begin{equation} \label{eq:A4.9}
\beta c_7  \abs{\delta_h \theta} \left(-\log\abs{h}\right)^{-1} \leq c_6,
\end{equation}
and
\begin{equation} \label{eq:A4.10}
\beta c_8 \norm{\theta_0}_{L^\infty} (-\log \abs{h})^{-1} \leq c_6,
\end{equation}
then \eqref{eq:A4.3} holds. Since $\abs{\delta_h \theta} \leq 2\norm{\theta_0}_{L^\infty}$, it suffices to show
\[-\log \abs{h} \geq \beta c_9 \norm{\theta_0}_{L^\infty}, \quad\text{for all } \abs{h} \leq r_0,\]
instead of \eqref{eq:A4.9} and \eqref{eq:A4.10}. We completes the proof by setting $M = 2(\beta c_9)^\beta \norm{\theta_0}_{L^\infty}^{1+\beta}$.
\end{proof}

\begin{proof}[Proof of Proposition \ref{prop:11}]
It suffices to prove the assertion for a smooth solution $\theta$. One may consider regularized equation with the term $-\epsilon \delta \theta$, get a uniform estimate independent of $\epsilon$, and then take the inviscid limit $\epsilon\to 0+$ to get the desired result.

\noindent{\it Step 1}. Evolution of $\abs{\nabla \theta}^2$\\
We take gradient on both sides of the equation to get:
\[\partial_t \nabla \theta + u \cdot \nabla^2 \theta + \nabla u \cdot \nabla \theta + \Lambda \nabla \theta = 0.\]
Multiply $\nabla\theta$ to both sides:
\[\nabla\theta\cdot \partial_t \nabla \theta + \nabla \theta \cdot u \cdot \nabla^2 \theta + \nabla \theta \cdot \nabla u \cdot \nabla \theta + \nabla \theta \cdot \Lambda \nabla \theta = 0.\]
Considering singular integral formulation for the nonlocal operator $\Lambda$ gives:
\begin{equation} \label{eq:A3.1}
\frac{1}{2} (\partial_t + u \cdot \nabla + \Lambda) \abs{\nabla\theta}^2 + \frac{1}{2}\underbrace{\frac{1}{2\pi}\int_{\R^2} \frac{\abs{\nabla \theta(x,t) - \nabla\theta(y,t)}^2}{\abs{x-y}^{3}} \,d y}_{=: D(x,t)} = -\nabla \theta \cdot \nabla u \cdot \nabla \theta.
\end{equation}
Our goal here is to prove
\[(\partial_t + u\cdot \nabla + \Lambda)\abs{\nabla\theta}^2(x,t) < 0\]
whenever $\abs{\nabla \theta(x,t)}$ is sufficiently large in terms of $M$.

\noindent{\it Step 2}. Pointwise lower bound of $D(x,t)$\\
A smooth function $\varphi:[0,\infty) \to \R$ is a non-decreasing cutoff function such that
\[\varphi(x) = 0 \text{ on } x\leq \frac{1}{2},\quad \varphi(x) = 1 \text{ on } x \geq 1,\quad 0\leq\varphi' \leq 4.\]
For some sufficiently small $R=R(x,t)>0$ which will be determined later, we have the following estimate on $D(x,t)$:
\begin{align*}
D(x,t) & \geq \frac{1}{2\pi} \int_{\R^2} \frac{\abs{\nabla \theta(x,t) - \nabla \theta(y,t)}^2}{\abs{x-y}^3}\varphi\Big(\frac{\abs{x-y}}{R}\Big) \,d y \\
& \geq \abs{\nabla\theta(x,t)}^2 \frac{1}{2\pi} \int_{\abs{x-y} \geq R} \frac{1}{\abs{x-y}^3} \,d y - 2 \nabla \theta(x,t) \cdot \frac{1}{2\pi} \int_{\R^2} \nabla \theta(y,t) \frac{1}{\abs{x-y}^3} \varphi\Big(\frac{\abs{x-y}}{R}\Big) \,d y \\
& \geq \abs{\nabla\theta(x,t)}^2 \int_R^\infty \frac{1}{\rho^2} \,d \rho - 2 \nabla \theta (x,t) \cdot \frac{1}{2\pi} \int_{\R^2} (\theta(y,t)-\theta(x,t)) \left\{ \nabla_v \left( \frac{1}{\abs{v}^3} \varphi \Big(\frac{\abs{v}}{R} \Big) \right)\right\}_{v=x-y}  \,d y \\
& \geq \frac{\abs{\nabla\theta(x,t)}^2}{\rho} - 2 \abs{ \nabla \theta (x,t)} \int_0^\infty \rho \Omega(\rho) \abs{ \frac{d}{d\rho} \Big(\frac{\varphi(\rho/R)}{\rho^3} \Big)} \,d \rho.
\end{align*}
The second term:
\[\int_0^\infty \rho \Omega(\rho) \abs{ \frac{d}{d\rho} \Big(\frac{\varphi(\rho/R)}{\rho^3} \Big)} \,d \rho \leq \frac{4}{R} \int_{R/2}^{R} \frac{\Omega(\rho)}{\rho^2} \,d\rho + 2\int_{R/2}^\infty \frac{\Omega(\rho)}{\rho^3} \,d\rho \leq \frac{c_1}{4} \frac{\Omega(R)}{R^2}.\]
Therefore,
\[D(x,t) \geq \frac{\abs{\nabla\theta(x,t)}^2}{R} - \abs{ \nabla \theta (x,t)} \frac{c_1 \Omega(R)}{2R^2} = \frac{\abs{\nabla\theta(x,t)}^2}{R} \left( 1 - \frac{c_1}{2\abs{\nabla\theta(x,t)}} \frac{\Omega(R)}{R} \right).\]
Set $R = R(x,t) >0$ to satisfy
\begin{equation} \label{eq:A3.2}
\frac{\Omega(R)}{R} = \frac{\abs{\nabla\theta(x,t)}}{c_1}.
\end{equation}
Then
\begin{equation} \label{eq:A3.3}
D(x,t) \geq \frac{\abs{\nabla\theta (x,t)}^2}{2R}.
\end{equation}
Note that $R(x,t)$ is arbitrarily small when $\abs{\nabla\theta(x,t)}$ is large enough.

\noindent{\it Step 3}. Estimate of $\nabla u$
\[\nabla u(x,t) = \nabla^\perp \Lambda^{-1} \nabla \theta(x,t) = P.V. \int_{\R^2} K(x-y) (\nabla \theta(x,t) - \nabla \theta(y,t)) \,d y,\]
where $K(z) = z^\perp / \abs{z}^3$.
We estimate $\nabla u$ by splitting $\R^2$ into two pieces; an inner piece $\abs{x-y} \leq R$ and an outer piece $\abs{x-y} > R$.

\noindent For the inner piece, we use the Cauchy-Schwartz inequality:
\begin{align*}
\abs{\nabla u_\text{in} (x,t)} & \leq C \sqrt{\bigg(\int_{\abs{x-y} \leq R} \frac{\abs{\nabla\theta(x,t) - \nabla\theta(y,t)}^2}{\abs{x-y}^3} \,d y \bigg)\bigg( \int_{\abs{x-y} \leq R} \frac{1}{\abs{x-y}} \,d y\bigg)} \\
& = c_3 \sqrt{D(x,t) R}.
\end{align*}
For the outer piece, we apply integration by parts and use the modulus of continuity of $\theta$:
\begin{align*}
\abs{\nabla u_\text{out}(x,t)} &\leq \abs{\int_{\abs{x-y} > R} \big(\theta(x,t) - \theta(y,t)\big) \nabla K(x-y) \,d y} \\
& \quad + \abs{\int_{\abs{x-y} = R} \big(\theta(x,t) - \theta(y,t)\big) K(x-y) \nu(y) \,d \sigma(y)} \\
&\leq C \int_{R}^\infty \frac{\Omega(\rho)}{\rho^2} \,d\rho + C \frac{\Omega(R)}{R} \leq c_4\frac{\Omega(R)}{R}.
\end{align*}
Therefore,
\begin{align} \label{eq:A3.4}
\begin{aligned}
\abs{\nabla u} \abs{\nabla \theta}^2 & \leq \abs{\nabla u_\text{in}} \abs{\nabla \theta}^2 + \abs{\nabla u_\text{out}} \abs{\nabla \theta}^2 \\
& \leq \frac{D}{4} + c_3^2  R \abs{\nabla \theta}^4 + c_4 \frac{\Omega(R)}{R} \abs{\nabla \theta}^2.
\end{aligned}
\end{align}

\noindent{\it Step 4}. Maximum principle\\
Combining \eqref{eq:A3.1}, \eqref{eq:A3.3}, and \eqref{eq:A3.4},
\begin{align*}
& (\partial_t + u \cdot \nabla + \Lambda) \abs{\nabla\theta}^2 + \frac{D}{2} + \frac{\abs{\nabla\theta}^2}{2R} \\
\leq &\, (\partial_t + u \cdot \nabla + \Lambda) \abs{\nabla\theta}^2 + D \\
\leq &\, 2\abs{\nabla u} \abs{\nabla \theta}^2 \leq \frac{D}{2} + 2c_3^2 R \abs{\nabla \theta}^4 + 2c_4 \frac{\Omega(R)}{R} \abs{\nabla\theta}^2.
\end{align*}
Hence
\[(\partial_t + u \cdot \nabla + \Lambda) \abs{\nabla\theta}^2 + \frac{\abs{\nabla\theta}^2}{2R} \leq 2c_3^2 R \abs{\nabla \theta}^4 + 2c_4 \frac{\Omega(R)}{R} \abs{\nabla\theta}^2.\]
Recall that \eqref{eq:A3.2} that
\[\abs{\nabla \theta} = c_1 \frac{\Omega(R)}{R}.\]
And let $r_1 \in (0,1)$ be 
\[M (-\log r_1)^{-\beta} = \Omega(r_1) = \min\Big\{\frac{1}{8^{1/2}c_1c_3}, \frac{1}{8c_4}, 1\Big\} =: c'.\]
If $\abs{\nabla \theta} > c_1 \frac{\Omega(r_1)}{r_1}$, then $R < r_1$,
\begin{equation*}
2c_3^2 R \bigg( c_1 \frac{\Omega(R)}{R} \bigg)^2 < \frac{1}{4R} ,\quad\text{and} \quad 2c_4 \frac{\Omega(R)}{R} < \frac{1}{4R}.
\end{equation*}
Hence,
\[(\partial_t + u\cdot \nabla + \Lambda) \abs{\nabla \theta}^2(x,t) < 0.\]
Due to the maximum principle,
\[\norm{\nabla \theta}_{L^\infty_{t,x}} \leq \max\Big\{c_1 \frac{\Omega(r_1)}{r_1} ,1\Big\} \leq \frac{C}{r_1} \leq C \exp \big( (c')^{-\frac{1}{\beta}} M^{\frac{1}{\beta}} \big). \qedhere\]
\end{proof}
\begin{proof}[Proof of Main Theorem 2] By setting $M = C\norm{\theta_0}_{L^\infty}^{1+\beta}$ for some large constant $C = C_\beta > 0$, we get \eqref{eq:A1.8} due to Proposition \ref{prop:10} and \ref{prop:11}:
\[\norm{\nabla \theta}_{L^\infty_{t,x}} \leq C \norm{\nabla\theta_0}_{L^\infty} \exp(C \norm{\theta_0}_{L^\infty}^{1+\frac{1}{\beta}}). \qedhere\]
\end{proof}

\begin{remark}
Since there is no control for constants as $\beta^{-1} \to 0$, we do not know how to improve it to $C\exp(C \norm{\theta_0}_{L^\infty})$.
\end{remark}

\begin{corollary} \label{cor:12}
Assume that $\theta_0 \in \mathcal{S}(\R^2)$. If
\[\alpha_1 + \alpha_2 < 1,\]
then there exists a global smooth solution $\theta$ of
\begin{equation} \label{eq:5.17}
\left\{\begin{array}{rl} \partial_t \theta + (u\cdot \nabla) \theta + \Lambda (\log^{-\alpha_1} \hspace{-.5ex} \Lambda) \theta & = 0, \\ u & = \nabla^\perp \Lambda^{-1} (\log^{\alpha_2} \hspace{-.3ex} \Lambda) \theta, \\ \theta(t=0) & = \theta_0. \end{array}\right.
\end{equation}
Moreover, the solution exhibits the following gradient estimate if $\norm{\nabla \theta_0}_{L^\infty} > 3 \norm{\theta_0}_{L^\infty}$:
\[\norm{\nabla \theta}_{L^\infty_{t,x}} \leq C \norm{\nabla \theta_0}_{L^\infty} \exp(C \norm{\theta_0}_{L^\infty}^{\gamma}),\]
where $\gamma > \frac{1}{1-\alpha_1 - \alpha_2}$ and the constants depend only on $\gamma$, $\alpha_1$, and $\alpha_2$.
\end{corollary}

\begin{proof}[Proof of Corollary \ref{cor:12}]
For a fixed $\beta > 0$, let $\omega(r) = (-\log r)^{-\beta}$ for $0<r<1$,\\
$\displaystyle A = A(\theta_0) = \frac{2\norm{\theta_0}_{L^\infty}}{\norm{\nabla \theta_0}_{L^\infty}}$ and set $\Omega(r) = \min\big\{M\omega (r/A), 2\norm{\theta_0}_{L^\infty}\big\}$ with a constant $M>0$ which will be determined later. Note that it suffices to show the case $A$ is smaller than a fixed constant $C_0$. Define $r_0 > 0$ to be
\[M\omega(r_0) = 2 \norm{\theta_0}_{L^\infty}.\]
If $M$ is large enough so that \eqref{eq:4.11} in the proof of Theorem \ref{thm:9} is satisfied for $\abs{h} \leq r_0$, then the modulus of continuity $\Omega$ is conserved. In the case of \eqref{eq:5.17},
\begin{equation*}
\frac{m(r^{-1}) \omega'(r)}{k(r) \omega(r)} = \beta \Big(-\log \frac{r}{A}\Big) ^{-1} \Big( \log(10 + r^{-1}) \Big)^{\alpha_1 + \alpha_2} \leq \frac{1}{c\norm{\theta_0}_{L^\infty}}, \quad\text{for } r\leq r_0.
\end{equation*}
Take $M = C \norm{\theta_0}_{L^\infty}^{1 + \frac{\beta}{1-\alpha_1 - \alpha_2}}$ with $C$ large enough, so that the above condition is satisfied. To be precise, set $r_0 > 0$ to be $\left\{ -\log \left( r_0/A\right) \right\}^{1-\alpha_1 - \alpha_2} = \max\{30^{1-\alpha_1 - \alpha_2}, c\beta\norm{\theta_0}_{L^\infty}\}$.

Let $r_1 > 0$ be $\displaystyle\Omega(r_1) = \min\bigg\{ \frac{c_2^{\frac{1}{2}}}{2^{\frac{1}{2}} c_1 c_3}, \frac{c_2}{2c_4}, 1 \bigg\}$, where $c_1,c_2,c_3,c_4$ are numeric constants in the proof of Theorem \ref{thm:8}. According to the proof of Theorem \ref{thm:8}, in particular the last part, we get
\[\norm{\nabla \theta}_{L^\infty} \leq c_1 \frac{\Omega(r_1)}{r_1} \leq \frac{C}{A} \exp\Big(C M^{\frac{1}{\beta}} \Big) \leq C \norm{\nabla\theta_0}_{L^\infty} \exp \Big(C \norm{\theta_0}_{L^\infty}^{\frac{1}{\beta} + \frac{1}{1-\alpha_1 - \alpha_2}} \Big). \qedhere\]
\end{proof}

\section{Asymptotic Decay of Solutions} \label{sec:6}
In this section, we prove Main Theorem 3, i.e., decay of Sobolev norms of solutions to slightly supercritical SQG equation \eqref{eq:1.8}. Recall that $L^\infty$-norm of smooth solutions to drift-diffusion equations decay as follows:
\begin{proposition}[Decay of $L^\infty$-norm, Theorem 4.1 in \cite{Cor04}] \label{prop:13}
Suppose $\theta$ and $u$ are smooth functions on $\R^2 \times [0,T)$ (or $\mathbb{T}^2 \times [0,T)$) satisfying
\[\left\{\begin{array}{rl} \partial_t \theta + u \cdot \nabla \theta + \Lambda^\alpha \theta &= 0, \\ \theta(t=0) &= \theta_0. \end{array} \right.\]
with $0<\alpha\leq 2$. Also assume that $\nabla \cdot u =0$ and $\theta \in L^\infty([0,T]; H^s)$ with $s>1$. Then
\[\norm{\theta(\cdot,t)}_{L^\infty} \leq \frac{\norm{\theta_0}_{L^\infty}}{(1+C \alpha t \norm{\theta_0}_{L^\infty}^\alpha)^{\frac{1}{\alpha}}},\quad 0\leq t < T.\]
The constant $C$ depends only on $\theta_0$.
\end{proposition}

\noindent\textbf{Main Theorem 3.} Consider a slightly supercritical SQG equation \eqref{eq:1.2} satisfying conditions of Main Theorem 1. Assume that the Fourier multiplier $P(\zeta) = P(\abs{\zeta})$ corresponding to the dissipation operator $\mathcal{L}$ satisfies
\[P(\zeta) \geq C k(\abs{\zeta}^{-1}) \geq C_\epsilon \abs{\zeta}^{\beta},\]
for any $0\leq \beta<1$. Suppose $\theta \in C([0,T];H^s)$ for $s>1$ is a global smooth solution to the equation, then
\[\norm{\theta(\cdot,t)}_{\dot{H}^{s}} \leq \exp(- c(t-t_0)) \norm{\theta(\cdot,t_0)}_{\dot{H}^s},\quad t>t_0,\]
where $t_0 > 0$ depends only on $\theta_0$ and $c>0$ depends only on $\theta_0$ and $s$.

\vspace{-.5\baselineskip}\begin{proof}
Fix $\epsilon_1$ and $\epsilon_2$ to be $0<\epsilon_2 < \epsilon_1 < \frac{1}{2}$. Set $\displaystyle\alpha_1 = \frac{\frac{1}{2} - \epsilon_1}{s+\frac{1}{2} - 3\epsilon_1}, \alpha_2 = \frac{\frac{1}{2} - \epsilon_2}{s+\frac{1}{2} - 3\epsilon_2}$ and choose $\alpha \in (\alpha_1, \alpha_2)$. Note that $0<\alpha<\frac{1}{2s+1}$. Let $p,q\in (2,\infty)$ be $p = \frac{2}{1-\alpha}$ and $q = \frac{1}{\alpha}$. Set $\delta_1, \delta_2 > 0$ to be
\[\alpha \delta_1 = \left( s+\frac{1}{2} - 3\epsilon_1 \right)(\alpha - \alpha_1),\quad (1-2\alpha) \delta_2 = \left( s+\frac{1}{2} - 3\epsilon_2 \right)(\alpha_2 - \alpha).\]
If we choose $\epsilon_1, \epsilon_2$ sufficiently close, then $\delta_1$ and $\delta_2$ can be arbitrary small.

Next, we perform the energy estimate by applying $\Lambda^{s}$ to the equation and multiplying by $\Lambda^s \theta$:
\begin{align*}
& \frac{1}{2} \frac{d}{dt} \norm{\Lambda^s \theta}_{L^2}^2 + \norm{\Lambda^{s} P^{\frac{1}{2}}(\Lambda) \theta}_{L^2}^2 = - \int \Lambda^s \big( (u\cdot \nabla) \theta \big) \Lambda^s \theta \,d x \\
&\qquad = - \int \Big( \Lambda^s \big( (u\cdot \nabla) \theta \big) - (u\cdot \nabla) \Lambda^s \theta \Big) \Lambda^s \theta \,d x - \frac{1}{2} \int (u\cdot \nabla) \abs{\Lambda^s \theta}^2 \,d x \\
&\qquad \leq \norm{\Lambda^s \big( (u\cdot \nabla) \theta \big) - (u\cdot \nabla) \Lambda^s \theta}_{L^{p^*}} \norm{\Lambda^s \theta}_{L^p}.
\end{align*}
The second term at the second line vanishes since div $u = 0$. Since $\frac{2}{p} + \frac{1}{q} = 1$ implies $\frac{1}{p^*} = \frac{1}{p} + \frac{1}{q}$, Kato-Ponce type commutator estimate gives the following. (See e.g. \cite{Kat88} or Lemma 7.2 of \cite{Don10})
\begin{align*}
\norm{\Lambda^s \big( (u\cdot \nabla) \theta \big) - (u\cdot \nabla) \Lambda^s \theta}_{L^{p^*}} &\lesssim \norm{\nabla u}_{L^q} \norm{\Lambda^s \theta}_{L^p} + \norm{\nabla \theta}_{L^q} \norm{\Lambda^s u}_{L^p} \\
&\lesssim \norm{\Lambda^{1+\delta_1} \theta}_{L^q} \norm{\Lambda^s \theta}_{L^p} + \norm{\Lambda \theta}_{L^q} \norm{\Lambda^{s+\delta_2} \theta}_{L^p}.
\end{align*}
Recall that $u = \mathcal{R}^\perp m(\Lambda) \theta$ where $m(\zeta) \leq C \abs{\zeta}^\epsilon$ for any $\epsilon>0$. The last inequalities are applications of Hörmander-Mikhlin type multiplier theorem. (See e.g. \cite{Ste93}) The Gagliardo–Nirenberg interpolation inequalities
\begin{align*}
\norm{\Lambda^{1+\delta_1} \theta}_{L^q} &\lesssim \norm{\Lambda^{2\epsilon_1 + \delta_1} \theta}_{L^\infty}^{1-2\alpha} \norm{\Lambda^{s+\frac{1}{2} - \epsilon_1} \theta}_{L^2}^{2\alpha}, \\
\norm{\Lambda^s \theta}_{L^p} &\lesssim \norm{\Lambda^{2\epsilon_1 + \delta_1} \theta}_{L^\infty}^{\alpha} \norm{\Lambda^{s+\frac{1}{2} - \epsilon_1} \theta}_{L^2}^{1-\alpha}, \\
\norm{\Lambda^{1} \theta}_{L^q} &\lesssim \norm{\Lambda^{2\epsilon_2 + 2\delta_2} \theta}_{L^\infty}^{1-2\alpha} \norm{\Lambda^{s+\frac{1}{2} - \epsilon_2} \theta}_{L^2}^{2\alpha}, \\
\norm{\Lambda^{s+\delta_2} \theta}_{L^p} &\lesssim \norm{\Lambda^{2\epsilon_2 + 2\delta_2} \theta}_{L^\infty}^{\alpha} \norm{\Lambda^{s+\frac{1}{2} - \epsilon_2} \theta}_{L^2}^{1-\alpha},
\end{align*}
follow from \cite{Mir18}. Summing up the estimates, we get
\[ \frac{1}{2} \frac{d}{dt} \norm{\Lambda^s \theta}_{L^2}^2 + \norm{\Lambda^{s} P^{\frac{1}{2}}(\Lambda) \theta}_{L^2}^2 \lesssim \big(\norm{\Lambda^{2\epsilon_1 + \delta_1} \theta}_{L^\infty} + \norm{\Lambda^{2\epsilon_2 + 2\delta_2} \theta}_{L^\infty} \big) \norm{\Lambda^{s} P^{\frac{1}{2}}(\Lambda) \theta}_{L^2}^2.\]
Repeating the exact same argument Proposition \ref{prop:13} was proved in \cite{Cor04}, we get
\[\lim_{t\to \infty}\norm{\theta(\cdot,t)}_{L^\infty} = 0.\]
Since $|\nabla \theta|$ is uniformly bounded, the interpolation give that $\norm{\Lambda^\gamma \theta(\cdot,t)}_{L^\infty} \to 0$ as $t\to \infty$ for $0<\gamma<1$. Hence, there exists $t_0 > 0$ so that
\[\frac{1}{2} \frac{d}{dt} \norm{\Lambda^s \theta}_{L^2}^2 \leq - c' \norm{\Lambda^{s} P^{\frac{1}{2}}(\Lambda) \theta}_{L^2}^2 \leq -c \norm{\Lambda^s \theta}_{L^2}^2,\quad \text{for all } t > t_0.\]
Therefore $\norm{\Lambda^s \theta(\cdot,t)}_{L^2} \leq \exp(-c(t-t_0)) \norm{\Lambda^s \theta(\cdot,t_0)}_{L^2}.$
\end{proof}

\begin{remark}
The Fourier multiplier $P$ considered in the statement of Main Theorem 3 do not contain our benchmark case $P(\zeta) = \zeta (\log(10+\abs{\zeta}))^{-\alpha_1}$, $\alpha_1 > 0$. In particular, the assumption that $P(\zeta) \geq c > 0$ was crucially used in the last line. The main issue is that we cannot prove exponential decay of low frequency part. We suggest two ways to overcome such problem.
\begin{enumerate}[label=(\roman*)]
	\item If the spatial domain is a bounded set $\Omega\subset \R^2$ or periodic $\mathbb{T}^2$ (with mean zero assumption $\int_{\mathbb{T}^2} \theta \,dx = 0$), then Poincaré inequality is applicable for the last line of the proof.
	\item It is possible to prove polynomial decay on $\mathbb{R}^2$ as follows:
	\[ \frac{1}{2} \frac{d}{dt} \norm{\Lambda^s \theta}_{L^2}^2 \leq -c' \norm{\Lambda^s P^{\frac{1}{2}}(\Lambda) \theta}_{L^2}^2 \leq -c'' \left(\frac{\norm{\Lambda^s \theta}_{L^2}^{1+\beta}}{\norm{\theta}_{L^2}^\beta}\right)^2, \quad\text{for all } t > t_0,\]
	for $0<\beta<\frac{1}{2s}$. Since $\norm{\theta(\cdot,t)}_{L^2} \leq \norm{\theta_0}_{L^2}$,
	\[\frac{d}{dt} \norm{\Lambda^s \theta}_{L^2} \leq -c\norm{\Lambda^s \theta}_{L^2}^{1+2\beta},\quad\text{for all } t>t_0.\]
	Therefore
	\[\norm{\Lambda^s \theta}_{L^2} \leq \frac{\norm{\Lambda^s \theta(t_0)}_{L^2}}{\big(1 + 2c\beta (t-t_0) \norm{\Lambda^s \theta(t_0)}_{L^2}^{2\beta} \big)^{-\frac{1}{2\beta}}},\quad \text{for all } t>t_0.\]
\end{enumerate}
\end{remark}

\subsection*{Acknowledgment} The author thanks In-Jee Jeong for providing valuable discussions. Also the author thanks the anonymous reviewers for their insightful comments, especially on the decay of solutions. On behalf of all authors, the corresponding author states that there is no conflict of interest.

\end{document}